\documentclass{article}%
\usepackage{amsmath}
\usepackage{amsfonts}
\usepackage{amssymb}
\usepackage{graphicx}%
\providecommand{\U}[1]{\protect\rule{.1in}{.1in}}
\newtheorem{theorem}{Theorem}
\newtheorem{acknowledgement}[theorem]{Acknowledgement}

\newtheorem{claim}[theorem]{Claim}

\newtheorem{corollary}[theorem]{Corollary}

\newtheorem{definition}[theorem]{Definition}

\newtheorem{lemma}[theorem]{Lemma}

\newtheorem{problem}[theorem]{Problem}
\newtheorem{proposition}[theorem]{Proposition}

\newenvironment{proof}[1][Proof]{\noindent\textbf{#1.} }{\ \rule{0.5em}{0.5em}}
\begin{document}

\date{}
\title{Ultrafilters in the random real model}
\author{Alan Dow \thanks{\textit{keywords: }\textsf{P}-points, Random forcing, strong
\textsf{P}-points, Gruff ultrafilters. \newline\textit{AMS Classification:}
03E05, 03E17, 03E35, 03E65.}
\and Osvaldo Guzm\'{a}n \thanks{The second author was supported by the PAPIIT grant
IA 104124 and the SECIHTI grant CBF2023-2024-903.}}
\maketitle

\begin{abstract}
We prove that \textsf{P}-points (even strong P-points) and Gruff ultrafilters
exist in any forcing extension obtained by adding fewer than $\aleph_{\omega}%
$-many random reals to a model of \textsf{CH. }These results improve and
correct previous theorems that can be found in the literature.

\end{abstract}

\section{Introduction}

Ultrafilters\footnote{Undefined concepts will be reviewed in later sections.}
play a fundamental role in infinite combinatorics, set-theoretic topology and
model theory. From constructing compactifications of topological spaces and
analyzing convergence, to proving Ramsey theorems, finding non-trivial
elementary embeddings of the universe, or building nonstandard models of a
theory, applications of ultrafilters are ubiquitous across these areas of
mathematics. Several important classes of ultrafilters on countable sets have
been introduced and studied over the years. A particularly notable example are
the \textsf{P}-points\emph{, }which were introduced by Walter Rudin in
\cite{Rudin}, to prove that under the Continuum Hypothesis (\textsf{CH}), the
space $\omega^{\ast}=$ $\beta\omega\setminus\omega$ is not homogeneous (the
same conclusion was later established without assuming \textsf{CH }by
Frol\'{\i}k in \cite{SumsofUltrafilters} and further refined by Kunen in
\cite{WeakPpointsin} who explicitly constructed ultrafilters with distinct
topological types). Since then, special classes of ultrafilters on countable
sets became a central topic of study and research.

\qquad\qquad\qquad

Although it is a straightforward theorem of \textsf{ZFC} that there are
(non-principal) ultrafilters on the natural numbers, the existence of
ultrafilters with interesting topological or combinatorial properties is far
more subtle. Moreover, their existence is often independent. The first major
result of this kind was obtained by Kunen in \cite{KunenSomePoints} where he
proved that Ramsey ultrafilters consistently do not exist. Some time later,
Miller proved in \cite{NoQpointsMiller} (see also \cite{Rat}) that
\textsf{Q}-points may not exist and Shelah constructed a model without
\textsf{P}-points (see \cite{Wimmers} and \cite{ProperandImproper}). More
recently, it was proved by Cancino and Zapletal (see \cite{IsbellProblem})
that it is consistent that every (non-principal) ultrafilter on $\omega$ is
Tukey top. For more results regarding the existence or non existence of
special ultrafilters, the reader may consult \cite{betaomega},
\cite{ShelahNowheredense}, \cite{Brendlenowheredense},
\cite{PseudoPpointsandSplittingNumber}, \cite{GruffUltrafilters},
\cite{TherearenoPpointsinSilverExtensions}, \cite{AboveFsigma} or
\cite{OnPospisilIdeals} among many others. As these results illustrate, the
existence of special classes of ultrafilters is a major concern for set
theorists and topologists.

\qquad\qquad

\emph{Random forcing }was introduced by Solovay in \cite{SolovayModelOriginal}
and models obtained by adding more than $\omega_{1}$ many random reals to a
model of \textsf{CH }are often called \emph{random models }(or \emph{random
real models}).\emph{ }Since random forcing is one of the most well-known and
studied forcing notions, one might expect that the structure of ultrafilters
in the random models is very well understood. However, this is far from the
case. The previously mentioned Theorem of Kunen in \cite{KunenSomePoints}
actually shows that there are no Ramsey ultrafilters in the random models. On
the other hand, there are \textsf{Q}-points in those models since the
dominating number is equal to $\omega_{1}$ (see \cite{HandbookBlass}). Now,
the important question is:\textit{ What about \textsf{P}-points? }This is
where the story gets complicated. In an unpublished note, Kunen proved that if
you add $\omega_{1}$ Cohen reals to a model of \textsf{CH }and then\textsf{
}any number of random reals, you will get a \textsf{P}-point. In particular,
\textsf{P}-points exists in some random models. The general case (without the
preliminary Cohen reals) was later addressed by Cohen\footnote{It is worth
pointing out that this is not the same Cohen who intoduced forcing and after
whom the Cohen reals are named.} \cite{PpointsinRandomUniverses}. He defined a
combinatorial object called \emph{pathway,} (which is very similar to Roitman%
\'{}%
s Model Hypothesis (\textsf{MH}), see
\cite{ParacompactnessinProductsOldandNew} and \cite{RotimanPrincipleHM})
proved that the existence of a pathway entails that there is a \textsf{P}%
-point and that pathways exist after adding any number of random reals to a
model of \textsf{CH. }Unfortunately, it was later discovered that the proof of
the existence of pathways is flawed\footnote{Despite of the mistake, the paper
\cite{PpointsinRandomUniverses} is very valuable. The introduction of pathways
is very important and the construction of a \textsf{P}-point from a pathway is
correct.} (see \cite{GruffCorrigendum} and
\cite{TherearenoPpointsinSilverExtensions}). No further progress was made on
this problem until the publication of \cite{PFiltersCohenRandomLaver}, where
the first author proved that there is a \textsf{P}-point if $\omega_{2}$-many
random reals are added to a model of \textsf{CH + }$\square_{\omega_{1}}$ (see
\cite{ParacompactnessinProductsOldandNew} for further results). In the present
work, we improve this result by showing that there are \textsf{P}-points (and
more) if less than $\aleph_{\omega}$ random reals are added to a model of the
Continuum Hypothesis (the principle \textsf{ }$\square_{\omega_{1}}$ is no
longer needed). The proof follows closely the argument from
\cite{PFiltersCohenRandomLaver}, \ but through a careful analysis of the
interaction of countable subsets within countable elementary submodels, we are
able to avoid the use of $\square_{\omega_{1}}$ and extend the result beyond
$\omega_{2}.$ We introduce a new type of combinatorial object, which we call a
\emph{multiple} $\mathfrak{d}$\emph{-pathway. }This notion has some
resemblance to Cohen's pathways, the hypothesis \textsf{MH} and to the
\emph{generalized pathways }introduced by Fern\'{a}ndez-Bret\'{o}n in
\cite{GeneralizedPathways}. We will prove that the existence of a multiple
$\mathfrak{d}$-pathway entails the existence of \textsf{P}-points (even strong
\textsf{P}-points) and Gruff ultrafilters. Finally, it will be proved that
multiple $\mathfrak{d}$-pathways exist if less than $\aleph_{\omega}$ random
or Cohen reals are added to a model of \textsf{CH.}

\qquad\qquad\qquad

The structure of the paper is as follows: after reviewing some notation and
preliminaries, in Section \ref{Seccion multiple pathways}, multiple
$\mathfrak{d}$-pathways will be introduced and we will prove some of their
most fundamental properties. In Section \ref{Ppoint} a \textsf{P}-point is
constructed from a multiple $\mathfrak{d}$-pathway. This construction will be
further refined in Section \ref{strong Ppoint} to get a strong \textsf{P}%
-point. Although in theory, the reader can skip Section \ref{Ppoint} and jump
to Section \ref{strong Ppoint}, we do not recommend it, since the construction
in Section \ref{Ppoint} is the best example to understand how to perform
transfinite recursions using a multiple $\mathfrak{d}$-pathway. Section
\ref{Gruff} contains our last application of multiple $\mathfrak{d}$-pathways,
the construction of a Gruff ultrafilter. In Section \ref{Seccion submodelos}
we develop some combinatorial results regarding countable elementary submodels
that will be needed later. In Section \ref{forzando pathways} we prove that
there are multiple $\mathfrak{d}$-pathways in the models obtained by adding
less than $\aleph_{\omega}$ many random or Cohen reals. Although we are mainly
interested in the random reals, the proof for Cohen reals is exactly the same.

\section{Notation}

For a set $X$, we denote by $\mathcal{P}\left(  X\right)  $ its power set. We
say that $\mathcal{F\subseteq}$ $\mathcal{P}\left(  X\right)  $ is a
\emph{filter on }$X$ if $X\in\mathcal{F}$ and $\emptyset\notin\mathcal{F},$
for every $A,B\subseteq X,$ if $A\in\mathcal{F}$ and $A\subseteq B$ then
$B\in\mathcal{F}$ and if $A,B\in\mathcal{F}$ then $A\cap B\in\mathcal{F}.$ \ A
family $\mathcal{I}\subseteq\mathcal{P}\left(  X\right)  $ is an \emph{ideal
on }$X$ if $\emptyset\in\mathcal{I}$ and $X\notin\mathcal{I},$ for every
$A,B\subseteq X,$ if $A\in\mathcal{I}$ and $B\subseteq A$ then $B\in
\mathcal{I}$ and if $A,B\in\mathcal{I}$ then $A\cup B\in\mathcal{I}.$ If
$\mathcal{B}$ is a family of subsets of $X,$ denote $\mathcal{B}^{\ast
}=\left\{  X\setminus B\mid B\in\mathcal{B}\right\}  .$ It is easy to see that
if $\mathcal{F}$ is a filter then $\mathcal{F}^{\ast}$ is an ideal (called the
\emph{dual ideal of }$\mathcal{F}$) and if $\mathcal{I}$ an ideal then
$\mathcal{I}^{\ast}$ is a filter (called the dual filter of $\mathcal{I}$). If
$\mathcal{I}$ is an ideal on $X$, define $\mathcal{I}^{+}=\mathcal{P}\left(
X\right)  \smallsetminus\mathcal{I}$, which is called the family of
$\mathcal{I}$\emph{-positive sets. }The \emph{Fr\'{e}chet filter }is the
filter of cofinite subsets$.$ An \emph{ultrafilter} is a maximal filter that
extends the Fr\'{e}chet filter (so in this work, all ultrafilters are
non-principal). If $\mathcal{U}$ is an ultrafilter, we say that
$\mathcal{B\subseteq U}$ \emph{is a base of }$\mathcal{U}$ if every element of
$\mathcal{U}$ contains one of $\mathcal{B}.$ We say that a family
$\mathcal{P\subseteq P}\left(  X\right)  $ is \emph{centered }if the
intersection of any finite collection of its elements is infinite.

\qquad\ \qquad\ \ \ \ 

By $\mathfrak{c}$ we denote the cardinality of the set of real numbers. For
any two sets $A$ and $B,$ we say $A\subseteq^{\ast}B$ ($A$ \emph{is an almost
subset of} $B$ or $A$ \emph{is almost contained in }$B$) if $A\setminus B$ is
finite. For $\mathcal{P\subseteq}$ $\left[  \omega\right]  ^{\omega}$ and
$A\in\mathcal{P}\left(  \omega\right)  ,$ we say that $A\ $\emph{is a
pseudointersection of }$\mathcal{P}$ if it is almost contained in all elements
of $\mathcal{P}.$ For $f,g\in\omega^{\omega}$, define $f\leq g$ if and only if
$f\left(  n\right)  \leq g\left(  n\right)  $ for every $n\in\omega$ and
$f\leq^{\ast}g$ if and only if $f\left(  n\right)  \leq g\left(  n\right)  $
holds for all $n\in\omega$ except finitely many. A family $\mathcal{B}%
\subseteq\omega^{\omega}$ is \emph{unbounded }if $\mathcal{B}$ is not bounded
with respect to $\leq^{\ast}.$ A family $\mathcal{D}\subseteq\omega^{\omega}$
is a \emph{dominating family }if for every $f\in\omega^{\omega},$ there is
$g\in\mathcal{D}$ such that $f\leq^{\ast}g.$ The \emph{bounding number
}$\mathfrak{b}$ is the size of the smallest unbounded family and the
\emph{dominating number }$\mathfrak{d}$ is the smallest size of a dominating
family. We say $\mathcal{S}=\left\{  f_{\alpha}\mid\alpha\in\mathfrak{b}%
\right\}  \subseteq\omega^{\omega}$ is a \emph{scale }if $\mathcal{S}$ is
dominating and $f_{\alpha}\leq^{\ast}f_{\beta}$ whenever $\alpha<\beta.$ It is
easy to see that $\mathfrak{b=d}$ is equivalent to the existence of a scale. A
function $f\in\omega^{\omega}$ is \emph{unbounded over a model }$M$ if
$f\nleq^{\ast}g$ for every $g\in M$ and \emph{is dominating over }$M$ if it
dominates every element of $M\cap\omega^{\omega}.$

\qquad\ \ \ \qquad\ \ \ \ 

$\mathcal{P}\left(  \omega\right)  $ will have its natural topology, which is
homeomorphic to $2^{\omega}$. In this way, the topology of $\mathcal{P}\left(
\omega\right)  $ has for a subbase the sets of the form $\left\langle
n\right\rangle _{0}=\left\{  A\subseteq\omega\mid n\notin A\right\}  $ and
$\left\langle n\right\rangle _{1}=\left\{  A\subseteq\omega\mid n\in
A\right\}  $, for $n\in\omega.$

\qquad\ \ \ \ 

A \emph{Polish space }is a separable and completely metrizable space. The
Baire space ($\omega^{\omega}$) and the Cantor space ($2^{\omega}$) are
examples of Polish spaces. \ We will need the concepts of $F_{\sigma},$
$G_{\delta},$ Borel, analytic, coanalytic and projective subsets of a Polish
space, which can be found in \cite{Kechris} or \cite{BorelSrivastava}. We say
that $T\subseteq2^{<\omega}$ is a \emph{tree }if it is closed under taking
initial segments and $f\in2^{\omega}$ is a \emph{branch of }$T$ if
$f\upharpoonright n\in T$ for every $n\in\omega.$ The set of all branches of
$T$ is denoted by $\left[  T\right]  $. In this context, a tree is
\emph{well-pruned }if every node can be extended to a branch. It is well known
that the compact subsets of $2^{\omega}$ correspond to branches of well-pruned
subtrees of $2^{<\omega}$ (see \cite{Kechris}). For $n\in\omega$, the
$n$\emph{-level of the tree} $T$ is denoted by $T_{n}.$

\qquad\qquad\qquad

Let $X$ be a topological space. We say that $C\subseteq X$ \emph{is crowded
}if it does not have an isolated point. A \emph{perfect subset of} $X$ is a
closed, non-empty crowded subset of $X.$ On the other hand, $S\subseteq X$
\emph{is scattered }if it does not contain a non-empty crowded subset.

\qquad\qquad\qquad\ \ 

We will work extensively with elementary submodels. The reader is invited to
consult \cite{AlanSubmodelos} for their most important properties and to learn
how to apply them in topology and set theory. For $\kappa$ a cardinal, by
\textsf{H}$\left(  \kappa\right)  $ we denote the collection of all subsets
whose transitive closure has size less than $\kappa.$ For $M$ a countable
elementary submodel of \textsf{H}$\left(  \kappa\right)  $ (and $\kappa
>\omega_{1}$), \emph{the height of }$M$ is $\delta_{M}=M\cap\omega_{1}.$ It is
easy to see that it is always a countable ordinal. We will fix
$\trianglelefteq$ a well order of \textsf{H}$\left(  \kappa\right)  $ and by
\textsf{Sub}$\left(  \kappa\right)  $ we denote the set of all countable
$M\subseteq$ \textsf{H}$\left(  \kappa\right)  $ such that $\left(
M,\in,\trianglelefteq\right)  $ is an elementary submodel of $($%
\textsf{H}$\left(  \kappa\right)  ,\in,\trianglelefteq).$ The well ordering
will play a key role in some of our arguments.

\qquad\ \ \ 

For $A$ a set of the ordinal numbers, we will denote by \textsf{OT}$\left(
A\right)  $ its order type. If $f$ is a function, by \textsf{dom}$\left(
f\right)  $ we denote its domain and \textsf{im}$\left(  f\right)  $ is its image.

\section{Forcing preliminaries \label{Preliminares Forcing}}

We review some preliminaries on forcing that will be needed in Section
\ref{forzando pathways}. Naturally, we assume the reader is already familiar
with the method of forcing as presented in \cite{oldKunen}.

\qquad\ \ \ 

Let $\mathbb{P}$ and $\mathbb{Q}$ be partial orders. By $V^{\mathbb{P}}$ we
denote the class of all $\mathbb{P}$-names as defined in \cite{oldKunen}. An
isomorphism $F:\mathbb{P\longrightarrow Q}$ can be extended recursively to a
bijection between $V^{\mathbb{P}}$ and $V^{\mathbb{Q}}$ (which we will also
denote as $F$) by letting $F\left(  \dot{a}\right)  =\{(F(\dot{b}),F\left(
p\right)  )\mid(\dot{b},p)\in\dot{a}\}.$ The proof of the following result can
essentially be found in \cite{Jech} or \cite{JechAC} and is easy to prove by
induction on the rank of names.

\begin{proposition}
Let $\mathbb{P}$ and $\mathbb{Q}$ be partial orders and
$F:\mathbb{P\longrightarrow Q}$ an isomorphism. For every $p\in\mathbb{P},$
$\varphi$ a formula, $\dot{a}_{0},...,\dot{a}_{n}\in V^{\mathbb{P}}$ and sets
$b_{0},...,b_{m},$ the following are equivalent:
\label{Prop formula isomorfismo}

\begin{enumerate}
\item $p\Vdash$\textquotedblleft$\varphi(\dot{a}_{0},...,\dot{a}_{n}%
,b_{0},...,b_{m})$\textquotedblright.

\item $F\left(  p\right)  \Vdash$\textquotedblleft$\varphi(F\left(  \dot
{a}_{0}\right)  ,...,F\left(  \dot{a}_{n}\right)  ,b_{0},...,b_{m}%
)$\textquotedblright.\ 
\end{enumerate}
\end{proposition}

\qquad\qquad\qquad\ \ 

We now review the standard method for adding many random or Cohen reals using
finite support. Proofs of the results mentioned below can be found in
\cite{KunenCohenyRandom}. For a set $I,$ we will always equip $2^{I}$ with its
usual Tychonoff topology (where $2=\left\{  0,1\right\}  $ is a discrete space).

\begin{definition}
Let $I,J$ be two infinite sets and $\triangle:I\longrightarrow J$ an injective function.

\begin{enumerate}
\item Define $\triangle^{\prime}:2^{J}\longrightarrow2^{I}$ given by
$\triangle^{\prime}\left(  f\right)  =f\circ\triangle$ for $f\in2^{J}.$

\item Define $\triangle_{\ast}:\mathcal{P}\left(  2^{I}\right)
\longrightarrow\mathcal{P}\left(  2^{J}\right)  $ where $\triangle_{\ast
}\left(  B\right)  =\left\{  f\in2^{J}\mid f\circ\triangle\in B\right\}  $ (in
other words, $\triangle_{\ast}\left(  B\right)  =\left(  \triangle^{\prime
}\right)  ^{-1}\left(  B\right)  $).
\end{enumerate}
\end{definition}

\qquad\qquad\ \ \ \ 

It is easy to see that if $\triangle:I\longrightarrow J$ is a bijection, then
$\triangle^{\prime}$ is a homeomorphism. The following proposition follows
from standard computations and diagram chasing arguments:

\begin{proposition}
Let $I,J,K$ be three infinite sets and $\triangle:I\longrightarrow J$,
$\sigma:J\longrightarrow K$ bijective functions. \label{prop estrellita}

\begin{enumerate}
\item $\left(  Id_{I}\right)  _{\ast}=Id_{2^{I}}$ (where $Id_{X}$ denotes the
identity mapping of a set $X$).

\item $\left(  \sigma\circ\triangle\right)  _{\ast}=\sigma_{\ast}%
\circ\triangle_{\ast}.$

\item $\left(  \triangle^{-1}\right)  _{\ast}=\left(  \triangle_{\ast}\right)
^{-1}.$

\item $\triangle_{\ast}$ is a bijection.
\end{enumerate}
\end{proposition}

\qquad\ \ \ \ 

We say $B\subseteq$ $2^{I}$ is \emph{Baire} if it belongs to the smallest
$\sigma$-algebra that contains all clopen sets. Denote by \textsf{Baire}%
$\left(  2^{I}\right)  $ the collection of all the Baire subsets of $2^{I}.$
If $I$ is countable, then the notion of Borel and Baire coincide, but they are
not the same if $I$ is uncountable. If $\triangle:I\longrightarrow J$ is
bijective, then $\triangle_{\ast}\upharpoonright$ \textsf{Baire}$\left(
2^{I}\right)  $ is a bijection between \textsf{Baire}$\left(  2^{I}\right)  $
and \textsf{Baire}$\left(  2^{J}\right)  .$

\begin{definition}
Let $I$ be an infinite set.

\begin{enumerate}
\item $\mathcal{M}_{I}$ denotes the $\sigma$-ideal of meager sets in $2^{I}.$

\item $\mathcal{N}_{I}$ denotes the $\sigma$-ideal of null sets in $2^{I}$
(where $2^{I}$ has the standard product measure).

\item Cohen forcing on $I$ (denoted by $\mathbb{C}\left(  I\right)  $) is the
set of all Baire subsets of $2^{I}$ that are not in $\mathcal{M}_{I}.$ We
order $\mathbb{C}\left(  I\right)  $ by inclusion: $A\leq B$ if and only if
$A\subseteq B$.

\item Random forcing on $I$ (denoted by $\mathbb{B}\left(  I\right)  $) is the
set of all Baire subsets of $2^{I}$ that are not in $\mathcal{N}_{I}.$ We
order $\mathbb{B}\left(  I\right)  $ by inclusion: $A\leq B$ if and only if
$A\subseteq B$.
\end{enumerate}
\end{definition}

\qquad\qquad\ \ 

Note that as we defined them, neither $\mathbb{C}\left(  I\right)  $ or
$\mathbb{B}\left(  I\right)  $ are separative. Alternatively, we could take
their quotients. This choice does not affect the content of the paper. If
$\triangle:I\longrightarrow J$ is bijective, then $\triangle_{\ast}%
:\mathbb{C}\left(  I\right)  \longrightarrow\mathbb{C}\left(  J\right)  $ and
$\triangle_{\ast}:\mathbb{B}\left(  I\right)  \longrightarrow\mathbb{B}\left(
J\right)  $ are isomorphism of partial orders. If $K\subseteq I,$ then
$\mathbb{B}\left(  K\right)  $ ($\mathbb{C}\left(  K\right)  $) is isomorphic
to a regular suborder of $\mathbb{B}\left(  I\right)  $ ($\mathbb{C}\left(
I\right)  $), and for convenience we will regard them as actual suborders.
Furthermore, if $\dot{a}$ is a $\mathbb{B}\left(  I\right)  $-name for a
subset of $\omega,$ we can find a countable $K\subseteq I$ such that $\dot{a}$
is a $\mathbb{B}\left(  K\right)  $-name (and the same for Cohen forcing). We
will be using all these facts implicitly.

\section{Multiple $\mathfrak{d}$-pathways \label{Seccion multiple pathways}}

Multiple $\mathfrak{d}$-pathways will be introduced in this section and we
will prove some of their most fundamental properties. Before proceeding, we
introduce some definitions.

\begin{definition}
Let $X$ be a set and $n\in\omega.$ A relation $R\subseteq X^{n}\times
\omega^{\omega}$ is $\leq^{\ast}$\emph{-adequate }if for every $x_{1}%
,...,x_{n}\in X,$ there is $f\in\omega^{\omega}$ such that for every
increasing $g\in\omega^{\omega},$ if $g\nleq^{\ast}f,$ then the relation
$R\left(  x_{1},...,x_{n},g\right)  $ holds.
\end{definition}

\qquad\qquad\ \ 

A function $f$ as above will be called an $R$\emph{-control for} $\left(
x_{1},...,x_{n}\right)  .$ It is not hard to find $\leq^{\ast}$-adequate
relations and several examples will be provided in the text.

\begin{definition}
Let $M_{0},...,M_{n}$ be countable elementary submodels of some \textsf{H}%
$\left(  \kappa\right)  .$ We will say that the sequence $\left\langle
M_{0},...,M_{n}\right\rangle $ is $\delta$\emph{-increasing }if $\delta
_{M_{i}}\leq\delta_{M_{i+1}}$ for each $i<n.$
\end{definition}

\qquad\qquad\ \ \ 

It is worth pointing out that in our work, the sequence $\left\langle
M_{0},...,M_{n}\right\rangle $ is typically not an $\in$-chain (which is often
the case when working with models as side conditions, see
\cite{PartitionProblems} and \cite{NotesonForcingAxioms}) and it will often be
the case that $\delta_{M_{i}}=\delta_{M_{i+1}}$ for some $i<n.$ When
discussing a specific $\delta$-increasing sequence$\ \left\langle
M_{0},...,M_{n}\right\rangle ,$ it will be understood that $n\in\omega$ and we
will write $\delta_{i}$ instead of $\delta_{M_{i}}$ in case there is no risk
of confusion.

\qquad\ \ \ 

We can finally introduce the main definition of this section:

\begin{definition}
Let $\kappa>\mathfrak{c}$ be a regular cardinal, $\mathcal{B}$ $=\left\{
f_{\alpha}\mid\alpha<\omega_{1}\right\}  \subseteq\omega^{\omega}$ a family of
increasing functions and $\mathcal{S\subseteq}$ \textsf{Sub}$\left(
\kappa\right)  $ a stationary subset of $[$\emph{H}$\left(  \kappa\right)
]^{\omega}$ such that every model in $\mathcal{S}$ has $\mathcal{B}$ as an
element. We say that $\left(  \mathcal{B},\mathcal{S}\right)  $ is a
\emph{multiple} $\mathfrak{d}$\emph{-pathway }if for every $\delta$-increasing
sequence $\left\langle M_{0},...,M_{n}\right\rangle $ of models from
$\mathcal{S},$ for every projective and $\leq^{\ast}$-adequate relation
$R\subseteq\left(  \omega^{\omega}\right)  ^{n+2}$ with $R\in M_{n}$ and
$x_{i}\in M_{i}$ for $i\leq n,$ we have that $R\left(  x_{0},...,x_{n}%
,f_{\delta_{n}}\right)  $ holds.
\end{definition}

\qquad\ \ \ 

Several comments regarding the definition are in order:

\begin{enumerate}
\item The term \textquotedblleft multiple\textquotedblright\ indicates that
there are typically several models that share the same height, which is not
the case in the Model Hypothesis of Roitman.

\item For our applications, the stationarity of $\mathcal{S}$ is only used to
ensure that every real (and hence every countable ordinal) appears in some
model of $\mathcal{S}.$

\item The relation $R$ is required to be projective. The key feature of this
requirement is that it can be coded by a real, so several strengthenings or
weakenings are possible. For the applications of \textsf{P}-points and strong
\textsf{P}-points, Borel relations are enough, but it appears that more is
needed for Gruff ultrafilters.

\item Following the approach of Fern\'{a}ndez-Bret\'{o}n in
\cite{GeneralizedPathways}, it is possible to define a notion of a multiple
pathway for more cardinal invariants of the continuum. We do not pursue this
approach here, since we do not have applications for other cardinal
invariants. However, the study of multiple pathways parametrized by cardinal
invariants of the continuum might prove fruitful in the future.
\end{enumerate}

\qquad\qquad\qquad\ \ \ 

To avoid constant repetition, when working with a multiple $\mathfrak{d}%
$-pathway $\left(  \mathcal{B},\mathcal{S}\right)  ,$ we will always use
$\left\{  f_{\alpha}\mid\alpha<\omega_{1}\right\}  $ to refer to the sequence
$\mathcal{B}$. We have the following simple result regarding multiple pathways:

\begin{proposition}
Let $\left(  \mathcal{B},\mathcal{S}\right)  $ be a multiple $\mathfrak{d}%
$-pathway. For every $M\in\mathcal{S},$ the function $f_{\delta_{M}}$ is
unbounded over $M.$ In particular, the existence of a multiple $\mathfrak{d}%
$-pathway implies that $\mathfrak{b}=\omega_{1}.$ \label{Escala}
\end{proposition}

\begin{proof}
Let $M\in\mathcal{S}$. Define the relation $R\subseteq\left(  \omega^{\omega
}\right)  ^{2}$ where $R\left(  g,f\right)  $ if $f$ $\nleq^{\ast}g.$ Clearly
$R\in M$, is a $G_{\delta}$ relation and is $\leq^{\ast}$-adequate. Since
obviously $\left\langle M\right\rangle $ is $\delta$-increasing, it follows
that every real in $M$ is $R$ related to $f_{\delta_{M}}.$ Finally, since the
models in a multiple $\mathfrak{d}$-pathway cover $\omega^{\omega},$ we
conclude that $\mathcal{B}$ is unbounded, hence $\mathfrak{b}=\omega_{1}.$
\end{proof}

\qquad\qquad\qquad\qquad

The definition of multiple $\mathfrak{d}$-pathway only mentions relations on
the Baire space, however, we can extend it to any Polish space, as the next
lemma shows:

\begin{lemma}
Let $\left(  \mathcal{B},\mathcal{S}\right)  $ be a multiple $\mathfrak{d}%
$-pathway and $\left\langle M_{0},...,M_{n}\right\rangle $ a $\delta
$-increasing sequence of models of $\mathcal{S}.$ For every Polish space $X\in
M_{0}\cap...\cap M_{n},$ $R\subseteq X^{n+1}\times\omega^{\omega}$ projective
and $\leq^{\ast}$-adequate relation in $M_{n}$ and $x_{i}\in X\cap M_{i}$ for
$i\leq n,$ the relation $R\left(  x_{0},...,x_{n},f_{\delta_{n}}\right)  $ holds.
\end{lemma}

\begin{proof}
We first find a continuous surjection $H:\omega^{\omega}\longrightarrow X$
with $H\in M_{0}\cap...\cap M_{n}$ (since $X$ is a Polish space, there is a
continuous surjection from $\omega^{\omega}$ to $X$, we now use the well order
of \textsf{H}$\left(  \kappa\right)  $ to find one, denoted $H$, that is in
all our models). Define the relation $P\subseteq\left(  \omega^{\omega
}\right)  ^{n+2}$ where $P\left(  g_{0},...,g_{n},f\right)  $ holds just in
case $R\left(  H\left(  g_{0}\right)  ,...,H\left(  g_{n}\right)  ,f\right)  $
is true. Note that $P\in M_{n}$ and is $\leq^{\ast}$-adequate. Moreover, it is
easy to see that $P$ is a continuous preimage of $R,$ so $P$ is projective as
well. $H$ is surjective and it is in every $M_{i},$ so we can find $g_{i}\in
M_{i}\cap\omega^{\omega}$ such that $H\left(  g_{i}\right)  =x_{i}.$ Since
$\left(  \mathcal{B},\mathcal{S}\right)  $ is a multiple $\mathfrak{d}%
$-pathway, we know that $P\left(  g_{0},...,g_{n},f_{\delta_{n}}\right)  $ is
true, which means that $R\left(  x_{0},...,x_{n},f_{\delta_{n}}\right)  $ is
true as well.
\end{proof}

\qquad\qquad

This covers the basics of multiple $\mathfrak{d}$-pathways and we are ready to
move on to applications.

\section{A P-point from a multiple $\mathfrak{d}$-pathway \label{Ppoint}}

An ultrafilter $\mathcal{U}$ on $\omega$ is a \textsf{P}\emph{-point }if every
countable subfamily of $\mathcal{U}$ has a pseudointersection in
$\mathcal{U}.$ Without a doubt, the class of \textsf{P}-points is among the
most important and studied families of ultrafilters on countable sets. Note
that ultrafilters that are not \textsf{P}-points are very easy to construct
(for example, every ultrafilter extending the dual filter of the density zero
ideal). In this way, the challenge is in constructing \textsf{P}-points.
Shelah was the first to show that it is consistent that there are no
\textsf{P}-points (see \cite{Wimmers} and \cite{ProperandImproper}). On the
other hand, several set theoretic axioms imply the existence of a
\textsf{P}-point, for example the equality $\mathfrak{d=c},$ the inequality
$\mathfrak{u<d}$ (see \cite{HandbookBlass}) or the parametrized diamond
$\Diamond\left(  \mathfrak{r}\right)  $ from \cite{ParametrizedDiamonds}. None
of these principles hold in the random real models, which makes the
construction of \textsf{P}-points in such models very interesting. We will now
build a \textsf{P}-point from a multiple $\mathfrak{d}$-pathway. Our approach
is inspired by Theorem 5.7 of \cite{PFiltersCohenRandomLaver}.

\begin{definition}
Let $F:\omega\longrightarrow\mathcal{P}\left(  \omega\right)  $ and
$g\in\omega^{\omega}.$ Define $F^{g}=%
{\textstyle\bigcup\limits_{n\in\omega}}
F\left(  n\right)  \cap g\left(  n\right)  .$
\end{definition}

\qquad\qquad\ \ \ 

It is clear that if $F$ is $\subseteq$-decreasing, then $F^{g}$ is a (possibly
finite) pseudointersection of \textsf{im}$\left(  F\right)  .$ Note that if
$f\leq g,$ then $F^{f}\subseteq F^{g}.$ It is trivial to see that if $F$ is
the constant function with value\ $A\subseteq\omega$ and $g\in\omega^{\omega}$
is increasing, then $F^{g}=A.$ The following results are easy and we leave
them to the reader.

\begin{lemma}
Let $F:\omega\longrightarrow\left[  \omega\right]  ^{\omega}$ be $\subseteq
$-decreasing. There is $f\in\omega^{\omega}$ such that for every increasing
$g\in\omega^{\omega},$ if $g\nleq^{\ast}f,$ then $F^{g}$ is infinite.
\label{adecuada P punto}\qquad\ \ \qquad\ \ \ 
\end{lemma}

\begin{lemma}
\label{Complejidad P punto}\qquad\ \ \ \ \qquad\qquad\ \ \ 

\begin{enumerate}
\item Let $n\in\omega.$ The set $R_{n}\subseteq\mathcal{P}\left(
\omega\right)  ^{n+1}$ consisting of all $\left(  A_{0},...,A_{n}\right)  $
such that $A_{0}\cap...\cap A_{n}$ is infinite, is $G_{\delta}.$

\item The set $R\subseteq\mathcal{P}\left(  \omega\right)  ^{\omega}$
consisting of all functions $F$ such that \textsf{im}$\left(  F\right)  $ is
centered, is $G_{\delta}.$
\end{enumerate}
\end{lemma}

\qquad\qquad\ \ 

We will need the following:

\begin{lemma}
Let $\left(  \mathcal{B},\mathcal{S}\right)  $ be a multiple $\mathfrak{d}%
$-pathway and $\left\langle M_{0},...,M_{n}\right\rangle $ a $\delta
$-increasing sequence of models of $\mathcal{S}$ and $m\leq n$ be the least
such that $\delta_{m}=\delta_{n}.$ For every $i\leq n,$ assume $F_{i}%
:\omega\longrightarrow\left[  \omega\right]  ^{\omega}\in M_{i}$ is
$\subseteq$-decreasing$.$

\qquad\ \ \ \qquad\ \ \ \ \ \ \ \ \qquad\ \ 

If $\left\{  F_{0}^{f_{\delta_{0}}},...,F_{m-1}^{f_{\delta_{m-1}}}\right\}
\cup$ \textsf{im}$\left(  F_{m}\right)  \cup...\cup$ \textsf{im}$\left(
F_{n}\right)  $ is centered, then $%
{\textstyle\bigcap\limits_{i\leq n}}
F_{i}^{f_{\delta_{i}}}$ is infinite. \label{Lema pathways para Ppunto}
\end{lemma}

\begin{proof}
Define the relation $R\subseteq\left(  \mathcal{P}\left(  \omega\right)
^{\omega}\right)  ^{n+1}\times\omega^{\omega}$ where $R(G_{0},...,G_{n},f)$
holds in case one of the following conditions is true:

\begin{enumerate}
\item $\left\{  G_{0}^{f_{\delta_{0}}},...,G_{m-1}^{f_{\delta_{m-1}}}\right\}
\cup$ \textsf{im}$\left(  G_{m}\right)  \cup...\cup$ \textsf{im}$\left(
G_{n}\right)  $ is not centered.

\item $G_{0}^{f_{\delta_{0}}}\cap...\cap G_{m-1}^{f_{\delta_{m-1}}}\cap
G_{m}^{f}\cap...\cap G_{n}^{f}$ is infinite.
\end{enumerate}

\qquad\ \ \ 

Since $f_{\delta_{0}},...,f_{\delta_{m-1}}\in M_{n},$ we get that $R\in
M_{n}.$ By Lemma \ref{Complejidad P punto} (or rather by its proof), the first
clause is an $F_{\sigma}$ condition and the second one is $G_{\delta},$ so $R$
is both $F_{\sigma\delta}$ and $G_{\delta\sigma},$ although we only care that
it is Borel. Moreover, it is $\leq^{\ast}$-adequate by Lemma
\ref{adecuada P punto} and Proposition \ref{Escala}. The conclusion follows
since $\left(  \mathcal{B},\mathcal{S}\right)  $ is a multiple $\mathfrak{d}$-pathway.
\end{proof}

\qquad\ \ \qquad\ \ 

We can now prove:

\begin{theorem}
If there is a multiple $\mathfrak{d}$-pathway, then there is a \textsf{P}%
-point. \label{Teorema P punto}
\end{theorem}

\begin{proof}
Fix $\left(  \mathcal{B},\mathcal{S}\right)  $ a multiple $\mathfrak{d}%
$-pathway. Define $D=\left\{  \delta_{M}\mid M\in\mathcal{S}\right\}  $ and
for every $\delta\in D,$ let $W_{\delta}=\cup\left\{  \mathcal{P}\left(
\omega\right)  \cap M\mid M\in\mathcal{S}\wedge\delta_{M}\leq\delta\right\}
.$ By recursion over $\delta\in D,$ we will define families $\mathcal{U}%
_{\delta},$ $\mathcal{A}_{\delta}$ and $\mathcal{P}_{\delta}$ with the
following properties:

\begin{enumerate}
\item $\mathcal{U}_{\delta},$ $\mathcal{P}_{\delta}$ and $\mathcal{A}_{\delta
}$ are subsets of $\left[  \omega\right]  ^{\omega}.$

\item $\mathcal{U}_{\gamma}\subseteq\mathcal{U}_{\delta}$ and $\mathcal{A}%
_{\gamma}\subseteq\mathcal{A}_{\delta}$ for $\gamma\in D\cap\delta.$

\item $\mathcal{A}_{\delta}\subseteq W_{\delta}.$

\item $\mathcal{P}_{\delta}$ is the collection of all $F^{f_{\delta}}$ for
which there is $M\in\mathcal{S}$ with $\delta_{M}=\delta,$ $F:\omega
\longrightarrow\mathcal{A}_{<\delta}$ is $\subseteq$-decreasing and belongs to
$M$ (where $\mathcal{A}_{<\delta}=%
{\textstyle\bigcup\limits_{\xi\in D\cap\delta}}
\mathcal{A}_{\xi}$).

\item $\mathcal{U}_{\delta}=%
{\textstyle\bigcup}
\{\mathcal{P}_{\gamma}\mid\gamma\in D\cap\left(  \delta+1\right)  \}.$

\item $\mathcal{U}_{\delta}\cup\mathcal{A}_{\delta}$ is centered.

\item $\mathcal{A}_{\delta}$ is $\subseteq$-maximal with respect to points 3
and 6.
\end{enumerate}

\qquad\qquad\ \ 

Assume we are at step $\delta\in D$ and $\mathcal{U}_{\gamma},$ $\mathcal{A}%
_{\gamma}$ and $\mathcal{P}_{\gamma}$ have been defined for all $\gamma\in
D\cap\delta.$ In case $\delta$ is the minimum of $D,$ we have $\mathcal{U}%
_{\delta}=\mathcal{P}_{\delta}=\emptyset.$ Choose $\mathcal{A}_{\delta
}\subseteq W_{\delta}$ any maximal centered set extending the Fr\'{e}chet
filter. Now consider the case where $\delta$ is not the least member of $D.$
Note that $\mathcal{U}_{\delta}$ and $\mathcal{P}_{\delta}$ are defined from
$\mathcal{A}_{<\delta},$ so we only need to find $\mathcal{A}_{\delta}.$
Define $\mathcal{U}_{<\delta}=%
{\textstyle\bigcup\limits_{\xi\in D\cap\delta}}
\mathcal{U}_{\xi}$ and note that $\mathcal{U}_{<\delta}\cup\mathcal{A}%
_{<\delta}$ is centered, since (by the recursion hypothesis) it is an
increasing union of centered sets. We now prove the following:

\begin{claim}
$\mathcal{U}_{\delta}\cup\mathcal{A}_{<\delta}$ is centered.
\end{claim}

\qquad\ \ \ 

Let $B_{0},...,B_{n}\in\mathcal{U}_{\delta}\cup\mathcal{A}_{<\delta}$, for
every $i\leq n,$ we find $M_{i}\in\mathcal{S}$ and $F_{i}\in M_{i}$ in the
following way:

\begin{enumerate}
\item In case $B_{i}\in\mathcal{A}_{<\delta},$ choose $M_{i}$ for which
$\delta_{i}=\delta_{M_{i}}<\delta$ and $B_{i}\in M_{i}.$ Let $F_{i}%
:\omega\longrightarrow\left[  \omega\right]  ^{\omega}$ be the constant
function with value $B_{i}.$

\item If $B_{i}\in\mathcal{U}_{\delta},$ choose $M_{i}$ with $\delta
_{i}=\delta_{M_{i}}\leq\delta$ and $F_{i}:\omega\longrightarrow\mathcal{A}%
_{<\delta_{i}}\in M_{i}$ that is $\subseteq$-decreasing and $B_{i}%
=F_{i}^{f_{\delta_{i}}}.$
\end{enumerate}

\qquad\ \ 

It might be possible that for some $i\leq n$ both clauses apply (in other
words, $B_{i}\in\mathcal{U}_{\delta}\cap\mathcal{A}_{<\delta}$). If that is
the case, we can choose to follow either one of them. For each $i\leq n,$ we
have the following:

\begin{enumerate}
\item $B_{i}=F_{i}^{f_{\delta_{i}}}.$

\item $F_{i}\in M_{i}$, is $\subseteq$-decreasing and \textsf{im}$\left(
F_{i}\right)  \subseteq$ $\mathcal{A}_{<\delta}.$
\end{enumerate}

\qquad\ \ 

By taking a reenumeration and possibly picking more elements of $\mathcal{U}%
_{\delta}\cup\mathcal{A}_{<\delta},$ we may assume that $\left\langle
M_{0},...,M_{n}\right\rangle $ is $\delta$-increasing and $\delta_{n}=\delta.$
Let $m\leq n$ be the least such that $\delta_{m}=\delta.$ We claim that
$H=\left\{  F_{0}^{f_{\delta_{0}}},...,F_{m-1}^{f_{\delta_{m-1}}}\right\}
\cup$ \textsf{im}$\left(  F_{m}\right)  \cup...\cup$ \textsf{im}$\left(
F_{n}\right)  $ is contained in $\mathcal{U}_{<\delta}\cup\mathcal{A}%
_{<\delta}.$ Pick $i\leq n.$ We have the following cases:

\begin{enumerate}
\item If $B_{i}\in\mathcal{A}_{<\delta},$ then $F_{i}^{\delta_{i}}=B_{i}$ so
there is nothing to do.

\item If $B_{i}\in\mathcal{U}_{\delta}$ and $i<m,$ then $F_{i}:\omega
\longrightarrow\mathcal{A}_{<\delta_{i}}$ which implies that $F_{i}%
^{f_{\delta_{i}}}\in\mathcal{U}_{\delta_{i}}\subseteq\mathcal{U}_{\delta}.$

\item If $B_{i}\in\mathcal{U}_{\delta}$ and $m\leq i,$ there is nothing to do
since we already noted that \textsf{im}$\left(  F_{i}\right)  \subseteq$
$\mathcal{A}_{<\delta}.$
\end{enumerate}

\qquad\ \ \ \qquad\ \ 

We already pointed out that $\mathcal{U}_{<\delta}\cup\mathcal{A}_{<\delta}$
is centered, so $H$ is centered as well. We are now in position to invoke
Lemma \ref{Lema pathways para Ppunto} and conclude that $%
{\textstyle\bigcap\limits_{i\leq n}}
F_{i}^{f_{\delta_{i}}}$ is infinite. Since $B_{i}=F_{i}^{f_{\delta_{i}}},$
this finishes the proof of the claim. We can now apply Zorn's Lemma and find
$\mathcal{A}_{\delta}\subseteq W_{\delta}$ extending $\mathcal{A}_{<\delta},$
such that its union with $\mathcal{U}_{\delta}$ is centered and it is maximal
with these properties.

\qquad\qquad\qquad

After completing the recursion, define $\mathcal{U=}%
{\textstyle\bigcup\limits_{\delta\in D}}
\mathcal{U}_{\delta}$ and $\mathcal{A=}%
{\textstyle\bigcup\limits_{\delta\in D}}
\mathcal{A}_{\delta}.$ We will now prove the following:

\begin{claim}
\qquad\ \ \ 

\begin{enumerate}
\item $\mathcal{U\cup A}$ is centered.

\item $\mathcal{U}=\mathcal{A}.$

\item $\mathcal{U}$ is an ultrafilter.

\item $\mathcal{U}$ is a \textsf{P}-point.
\end{enumerate}
\end{claim}

\qquad\ \ \ 

To see the first point, simply note that $\mathcal{U\cup A}$ is the increasing
union of the sets $\mathcal{U}_{\delta}\cup\mathcal{A}_{\delta}$ and since we
already knew those are centered, it follows that $\mathcal{U\cup A}$ is
centered as well. We will now prove $\mathcal{U\subseteq A}.$ Let
$B\in\mathcal{U},$ find $\delta\in D$ such that $B\in W_{\delta}$ (recall that
$\mathcal{S}$ is stationary). Since $B\in\mathcal{U},$ clearly $\mathcal{U}%
_{\delta}\cup\left\{  B\right\}  \cup\mathcal{A}_{\delta}$ is centered$.$ By
the maximality of $\mathcal{A}_{\delta}$ that $B\in\mathcal{A}_{\delta}.$ We
will now show that $\mathcal{A\subseteq U}.$ Let $A\in\mathcal{A}$ and
$\delta\in D$ such that $A\in\mathcal{A}_{\delta}.$ Since $\mathcal{S}$ is
stationary, we can find $\gamma\in D$ and $M\in\mathcal{S}$ such that
$\delta_{M}=\gamma$ and $\delta,A\in M$ (so $\delta<\gamma$). Let $F$ be the
constant sequence with value $A.$ We get that $A=F^{f_{\gamma}}\in
\mathcal{P}_{\gamma},$ so $B\in\mathcal{U}.$

\qquad\qquad\qquad\qquad

It is time to prove that $\mathcal{U}$ is an ultrafilter. Let
$A,B,C,E\subseteq\omega$ such that $A,B\in\mathcal{U}$ and $A\subseteq C.$
Find $\delta\in D$ and $M\in\mathcal{S}$ such that $\delta_{M}=\delta$ and
$A,B,C,E\in M.$ In this way, $A\cap B,C,E$ and $\omega\setminus E$ are in
$W_{\delta}.$ Since $\mathcal{U=A}$ is centered, it follows that
$\mathcal{U}_{\delta}\cup\left\{  A\cap B,C\right\}  \cup\mathcal{A}_{\delta}$
is centered. By the maximality of $\mathcal{A}_{\delta},$ we get that $A\cap
B,C\in\mathcal{A}_{\delta}.$ Moreover, either $\mathcal{U}_{\delta}%
\cup\mathcal{A}_{\delta}\cup\left\{  E\right\}  $ is centered or
$\mathcal{U}_{\delta}\cup\mathcal{A}_{\delta}\cup\left\{  \omega\setminus
E\right\}  $ is centered, so the maximality of $\mathcal{A}_{\delta}$ entails
that $E\in\mathcal{A}_{\delta}$ or $\omega\setminus E\in\mathcal{A}_{\delta}.$
Finally, recall that in the first step of the recursion we made sure that
$\mathcal{U}$ contains the Fr\'{e}chet filter.

\qquad\qquad\qquad

It remains to prove that $\mathcal{U}$ is a \textsf{P}-point. Pick
$F:\omega\longrightarrow\mathcal{U}$ that is $\subseteq$-decreasing$.$ For
every $n\in\omega,$ choose $\delta_{n}\in D$ such that $F\left(  n\right)
\in\mathcal{A}_{\delta_{n}}.$ We now find $\delta\in D$ and $M\in\mathcal{S}$
such that $\delta_{M}=\delta$ and $F,\left\{  \delta_{n}\mid n\in
\omega\right\}  \in M.$ In this way, $F\in M$ and it is a decreasing sequence
of elements of $\mathcal{A}_{<\delta},$ so $F^{f_{\delta}}\in\mathcal{P}%
_{\delta}\subseteq\mathcal{U}.$
\end{proof}

\section{A strong P-point from a multiple $\mathfrak{d}$-pathway
\label{strong Ppoint}}

A \textsf{P}-point is an ultrafilter for which we can diagonalize against
countably many of its elements, while a strong \textsf{P}-point is an
ultrafilter for which we can diagonalize against countably many of its compact
subsets. Since every singleton is compact, it follows that strong
\textsf{P}-points are \textsf{P}-points. It is not difficult to show that
Ramsey ultrafilters are never strong \textsf{P}-points. The mere existence of
a \textsf{P}-point is not enough to imply either the existence of a strong
\textsf{P}-point, nor of a \textsf{P}-point that is not strong\footnote{Note
that a weak \textsf{P}-point is not the same as a \textsf{P}-point that is not
strong.}. In the Miller model every \textsf{P}-point is strong, while in
Shelah's model of only one Ramsey ultrafilter (see \cite{ProperandImproper}),
every \textsf{P}-point is Ramsey, and therefore not strong. Naturally, both
types of \textsf{P}-points exist under \textsf{CH }(\textsf{cov}$\left(
\mathcal{M}\right)  =\mathfrak{c}$ is enough). In \cite{HrusakVerner}, Blass,
Hru\v{s}\'{a}k and Verner proved that strong \textsf{P}-points are precisely
the ultrafilters whose Mathias forcing does not add dominating reals. This
makes them very useful for constructing models where the bounding number is
small, while other invariants like the splitting number or variants of the
almost disjointness number are large (see \cite{BrendleRaghavan},
\cite{OnCardinalInvariantsoftheContinuum}, \cite{ProperandImproper},
\cite{TotallyMAD}, \cite{MadFamiliesSplittingFamiliesandLargeContinuum},
\cite{MobandMAD}, \cite{CanjarFilters}, \cite{CanjarFiltersII} or
\cite{uanda}). Strong \textsf{P}-points were introduced by Laflamme in
\cite{LaflammeCompleteCombinatorics} with the purpose of studying the
collection of all $F_{\sigma}$ filters as a forcing notion. We will obtain a
strong \textsf{P}-point from a multiple $\mathfrak{d}$-pathway.

\begin{definition}
An ultrafilter $\mathcal{U}$ on $\omega$ is a \emph{strong P-point} if for
every sequence $\left\langle \mathcal{C}_{n}\right\rangle _{n\in\omega}$ of
compact subsets of $\mathcal{U},$ there is a partition $P=\left\{  P_{n}\mid
n\in\omega\right\}  $ of $\omega$ into intervals such that the set
$\{A\subseteq\omega\mid\forall n\in\omega\exists X_{n}\in\mathcal{C}%
_{n}\left(  A\cap P_{n}=X_{n}\cap P_{n}\right)  \}$ is contained in
$\mathcal{U}.$
\end{definition}

\qquad\ \ 

For convenience, we will denote \textsf{fin }$=\left[  \omega\right]
^{<\omega}\setminus\left\{  \emptyset\right\}  .$ Given an ultrafilter
$\mathcal{F}$ on $\omega,$ we define the filter $\mathcal{F}^{<\omega}$ on
\textsf{fin }that is generated by $\{\left[  A\right]  ^{<\omega}%
\setminus\left\{  \emptyset\right\}  \mid A\in\mathcal{F}\}.$ It is not hard
to see for $X\subseteq$ \textsf{fin, }we have that $X\in\left(  \mathcal{F}%
^{<\omega}\right)  ^{+}$ if and only if for every $A\in\mathcal{F},$ there is
$s\in X$ such that $s\subseteq A.$ The following theorem combines results from
Blass, Chodounsk\'{y}, Hru\v{s}\'{a}k, Minami, Repov\v{s}\textsf{, }Verner and
Zdomskyy (see \cite{HrusakVerner},
\cite{MathiasPrikryandLaverPrikryTypeForcing} and
\cite{MathiasForcingandCombinatorialCoveringPropertiesofFilters}). We will
only need the equivalence between 1) and 3).

\begin{theorem}
Let $\mathcal{U}$ be an ultrafilter on $\omega.$ The following are equivalent:
\label{Equivalencias strong ppoint}

\begin{enumerate}
\item $\mathcal{U}$ is a strong \textsf{P}-point.

\item The Mathias forcing of $\mathcal{U}$ does not add dominating reals.

\item $\mathcal{U}^{<\omega}$ is a \textsf{P}$^{+}$ filter, this means that
every countable $\subseteq$-decreasing sequence contained in $\left(
\mathcal{U}^{<\omega}\right)  ^{+}$ has a pseudointersection in $\left(
\mathcal{U}^{<\omega}\right)  ^{+}.$

\item $\mathcal{U}$ is a Menger subset of $\mathcal{P}\left(  \omega\right)
.$
\end{enumerate}
\end{theorem}

\qquad\ \ 

We now need to prove some results regarding compact sets that will be used later.

\begin{definition}
Let $X\subseteq$ \textsf{fin. }Define:

\begin{enumerate}
\item $\mathcal{C}\left(  X\right)  =\left\{  A\subseteq\omega\mid\forall s\in
X\left(  A\cap s\neq\emptyset\right)  \right\}  $.

\item $T_{X}\subseteq2^{<\omega}$ is the set consisting of all $s\in
2^{<\omega}$ such that for every $u\in X,$ if $u\subseteq$ \textsf{dom}%
$\left(  s\right)  ,$ then $u\cap s^{-1}\left(  \left\{  1\right\}  \right)
\neq\emptyset.$
\end{enumerate}
\end{definition}

\qquad\ \ 

It is easy to see that $\mathcal{C}\left(  X\right)  $ is a compact subset of
$\mathcal{P}\left(  \omega\right)  ,$ $T_{X}$ is a well-pruned tree and
$\mathcal{C}\left(  X\right)  =\left[  T_{X}\right]  $ (we are identifying a
set with its characteristic function). \ The relevance of these notions is the
following lemma, which can be found in \cite{uanda} or \cite{CanjarFiltersII}:

\begin{lemma}
Let $\mathcal{U}$ be an ultrafilter on $\omega$ and $X\subseteq$ \textsf{fin.
}$X\in\left(  \mathcal{U}^{<\omega}\right)  ^{+}$ if and only if
$\mathcal{C}\left(  X\right)  \subseteq\mathcal{U}.$
\label{Lema positivo ultrafiltro}
\end{lemma}

\qquad\ \ \ \ \ 

We now prove the following:

\begin{lemma}
Let $X,Y\subseteq$ \textsf{fin. }The following are equivalent:

\begin{enumerate}
\item $\mathcal{C}\left(  X\right)  \cup\mathcal{C}\left(  Y\right)  $ is centered.

\item For every $n,m,l\in\omega$ and every $s_{1},...,s_{n}\in\left(
T_{X}\right)  _{m}$, $z_{1},...,z_{n}\in\left(  T_{Y}\right)  _{m}$ there is
$k>m,l$ such that for every $\overline{s}_{1},...,\overline{s}_{n}\in\left(
T_{X}\right)  _{k}$, $\overline{z}_{1},...,\overline{z}_{n}\in\left(
T_{Y}\right)  _{k}$ for which $s_{i}\subseteq\overline{s}_{i}$ and
$z_{i}\subseteq\overline{z}_{i}$ for every $i\leq n,$ there is $u\in\lbrack
l,k)$ such that $\overline{s}_{i}\left(  u\right)  =\overline{z}_{i}\left(
u\right)  =1$ for every $i\leq n.$
\end{enumerate}
\end{lemma}

\begin{proof}
It is easy to see that 2) implies 1). Now assume that $\mathcal{C}\left(
X\right)  \cup\mathcal{C}\left(  Y\right)  $ is centered. For each $i\leq n,$
define $S_{i}=\{A\in\mathcal{C}\left(  X\right)  \mid A\cap m=s_{i}%
^{-1}\left(  \left\{  1\right\}  \right)  \}$ and $Z_{i}=\{B\in\mathcal{C}%
\left(  X\right)  \mid B\cap m=z_{i}^{-1}\left(  \left\{  1\right\}  \right)
\}.$ These are compact subsets of $\mathcal{P}\left(  \omega\right)  .$ We now
consider the compact space $W=S_{1}\times...\times S_{n}\times Z_{1}%
\times...\times Z_{n}.$ For $u\geq l,$ define $V_{u}=\{\left(  A_{1}%
,...,A_{n},B_{1},...,B_{n}\right)  \in W\mid u\in A_{1}\cap...\cap A_{n}\cap
B_{1}\cap...\cap B_{n}\}.$ Each $V_{u}$ is open. Moreover, since
$\mathcal{C}\left(  X\right)  \cup\mathcal{C}\left(  Y\right)  $ is centered,
it follows that $\left\{  V_{u}\mid u\geq l\right\}  $ is an open cover of
$W.$ Since $W$ is compact, we can find $r\in\omega$ such that $\left\{
V_{l},...,V_{r}\right\}  $ is an open cover of $W.$ We claim that $k=r+1$ is
as desired. Indeed, let $\overline{s}_{1},...,\overline{s}_{n}\in\left(
T_{X}\right)  _{k}$ and $\overline{z}_{1},...,\overline{z}_{n}\in\left(
T_{Y}\right)  _{k}$ for which $s_{i}\subseteq\overline{s}_{i}$ and
$z_{i}\subseteq\overline{z}_{i}$ for every $i\leq n.$ Pick $\left(
A_{1},...,A_{n},B_{1},...,B_{n}\right)  \in W$ such that $A_{i}\cap
k=\overline{s}_{i}^{-1}\left(  \left\{  1\right\}  \right)  $ and $B\cap
m=\overline{z}_{i}^{-1}\left(  \left\{  1\right\}  \right)  $ for every $i\leq
n$ (this is possible since the trees are well-pruned). In this way, there is
$u\in\lbrack l,k)$ such that $\left(  A_{1},...,A_{n},B_{1},...,B_{n}\right)
\in V_{u},$ which implies the desired conclusion.
\end{proof}

\qquad\qquad\qquad\ \ \ \qquad\ \ \ \ \qquad\ \ 

With the lemma (or rather its natural generalization to any finite number of
subsets of \textsf{fin}), we get:

\begin{corollary}
Let $n\in\omega$ and define $R\subseteq\mathcal{P}\left(  \left[
\omega\right]  ^{<\omega}\right)  ^{n}$ as the set of all $\left(
X_{1},...,X_{n}\right)  $ such that $%
{\textstyle\bigcup\limits_{i\leq n}}
\mathcal{C}\left(  X_{i}\right)  $ is centered. $R$ is a $G_{\delta}$
relation. \label{complejidad strong ppoint}
\end{corollary}

We now introduce the following notion:

\begin{definition}
For a function $F:\omega$ $\longrightarrow$ $\mathcal{P(}$\textsf{fin}%
$\mathsf{)}$ and $g\in\omega^{\omega},$ we define the set $F^{g}=%
{\textstyle\bigcup\limits_{n\in\omega}}
F\left(  n\right)  \cap\mathcal{P}\left(  g\left(  n\right)  \right)  .$
\end{definition}

If $F$ is $\subseteq$-decreasing, then $F^{g}$ is a pseudointersection of
\textsf{im}$\left(  F\right)  .$ Note that if $f\leq g,$ then $F^{f}\subseteq
F^{g}.$ It is trivial to see that if $F$ is the constant function with value
$A\subseteq$ \textsf{fin} and $f$ is increasing, then $F^{f}=A.$ Lastly, if
$B\subseteq\omega,$ then $\mathcal{C(}\left[  B\right]  ^{1})$ consists of all
supersets of $B.$ The following lemma has appeared in \cite{uanda} (Lemma 57)
and in \cite{CanjarFiltersII} (Lemma 3):

\begin{lemma}
Let $\mathcal{F}$ be a filter, $\mathcal{D\subseteq F}$ compact and
$X_{1},...,X_{n}\subseteq$ $\mathcal{P(}$\textsf{fin}$\mathsf{)}$ with
$\mathcal{C}\left(  X_{1}\right)  ,...,\mathcal{C}\left(  X_{n}\right)
\subseteq\mathcal{F}.$ For every $i\leq n,$ there is $Y_{i}\in\left[
X_{i}\right]  ^{<\omega}$ such that for every $F\in\mathcal{D}$ and $A_{i}%
^{1},....,A_{i}^{n}\in\mathcal{C}\left(  Y_{i}\right)  ,$ we have that $F\cap%
{\textstyle\bigcap\limits_{i,j\leq n}}
A_{i}^{j}\neq\emptyset.$ \label{Lema para adecuada strong p point}
\end{lemma}

It will be convenient to introduce the following notion:

\begin{definition}
Let $\mathcal{A}$ be a family of compact subsets of $\mathcal{P}\left(
\omega\right)  .$ We say that $\mathcal{A}$ \emph{is compatible }if $%
{\textstyle\bigcup}
\mathcal{A}$ is centered.
\end{definition}

\qquad\qquad\ \ 

We can now prove the following:

\begin{lemma}
Let $\mathcal{D\subseteq P}\left(  \omega\right)  $ be a compact set,
$n\in\omega$ and $F_{i}:\omega$ $\longrightarrow$ $\mathcal{P(}$%
\textsf{fin}$\mathsf{)}$ for each $i\leq n$ such that $\left\{  \mathcal{D}%
\right\}  \cup\left\{  \mathcal{C}\left(  F_{i}\left(  k\right)  \right)  \mid
i\leq n\wedge k\in\omega\right\}  $ is compatible. There is $h\in
\omega^{\omega}$ such that for every increasing $g\in\omega^{\omega},$ if
$g\nleq^{\ast}h,$ then $\mathcal{D\cup C}\left(  F_{0}^{g}\right)  \cup
...\cup$ $\mathcal{C}\left(  F_{n}^{g}\right)  $ is compatible.
\label{adecuada para strong ppoint}
\end{lemma}

\begin{proof}
We may assume that $F_{i}\left(  k\right)  \subseteq\left[  \omega\setminus
k\right]  ^{<\omega}$ for every $k\in\omega$ and $i\leq n.$ With the aid of
Lemma \ref{Lema para adecuada strong p point}, we can find an increasing
$h\in\omega^{\omega}$ such that for every $m\in\omega$ the following holds:
For every $E_{0},...,E_{m}\in\mathcal{D}$ and $A_{i}^{1},...,A_{i}^{m}\in
F_{m}\left(  i\right)  \cap\mathcal{P}\left(  h\left(  m\right)  \right)  $
for every $i\leq n,$ we have that $E_{0}\cap...\cap E_{m}\cap%
{\textstyle\bigcap\limits_{i\leq n,\text{ }j\leq m}}
A_{i}^{j}$ is non-empty. It is easy to see that $h$ has the desired property.
\end{proof}

\qquad\qquad\ \ 

We only need one more preliminary result:

\begin{lemma}
Let $\left(  \mathcal{B},\mathcal{S}\right)  $ be a multiple $\mathfrak{d}%
$-pathway, $\left\langle M_{0},...,M_{n}\right\rangle $ a $\delta$-increasing
sequence of models in $\mathcal{S}$ and $m\leq n$ the least such that
$\delta_{m}=\delta_{n}.$ For every $i\leq n,$ let $F_{i}:\omega$
$\longrightarrow$ $\mathcal{P(}$\textsf{fin}$\mathsf{)}\in M_{i}$ that is
$\subseteq$-decreasing\textsf{.}

\qquad\ \qquad\ \ \ \ \ \ \ \ \ \ 

If $\{\mathcal{C(}F_{0}^{f_{\delta_{0}}}),...,\mathcal{C(}F_{m-1}%
^{f_{\delta_{m-1}}})\}\cup\left\{  \mathcal{C}\left(  F_{i}\left(  k\right)
\right)  \mid m\leq i\leq n\wedge k\in\omega\right\}  $ is compatible, then $%
{\textstyle\bigcup\limits_{i\leq n}}
\mathcal{C}(F_{i}^{f_{\delta_{i}}})$ is centered.
\label{Lema pathway strong ppoint}
\end{lemma}

\begin{proof}
Define the relation $R\subseteq(\mathcal{P(}$\textsf{fin}$)^{\omega}%
)^{n+1}\times\omega^{\omega}$ where $R\left(  G_{0},...,G_{n},f\right)  $
holds just in case one of the following conditions is met:

\begin{enumerate}
\item $\{\mathcal{C(}G_{0}^{f_{\delta_{0}}}),...,\mathcal{C(}G_{m-1}%
^{f_{\delta_{m-1}}})\}\cup\left\{  \mathcal{C(}G_{i}\left(  k\right)  )\mid
m\leq i\leq n\wedge k\in\omega\right\}  $ is not compatible.

\item $\mathcal{C(}G_{0}^{f_{\delta_{0}}})\cup...\cup\mathcal{C(}%
G_{m-1}^{f_{\delta_{m-1}}})\cup\mathcal{C(}G_{m}^{f})\cup...\cup
\mathcal{C(}G_{n}^{f})$ is centered.
\end{enumerate}

\qquad\ \ \ 

Since $f_{\delta_{0}},...,f_{\delta_{m-1}}\in M_{n},$ we conclude that $R\in
M_{n}.$ By Corollary \ref{complejidad strong ppoint}, the first clause is an
$F_{\sigma}$ condition and the second one is $G_{\delta},$ so $R$ is both
$F_{\sigma\delta}$ and $G_{\delta\sigma},$ hence it is Borel. Moreover, it is
$\leq^{\ast}$-adequate by Lemma \ref{adecuada para strong ppoint}. The
conclusion of the lemma follows since $\left(  \mathcal{B},\mathcal{S}\right)
$ is a multiple $\mathfrak{d}$-pathway.
\end{proof}

\qquad\qquad\ \ 

We now prove the main result of the section:

\begin{theorem}
If there is a multiple $\mathfrak{d}$-pathway, then there is a strong
\textsf{P}-point. \label{Teorema Strong}
\end{theorem}

\begin{proof}
Let $\left(  \mathcal{B},\mathcal{S}\right)  $ be a multiple $\mathfrak{d}%
$-pathway. Define $D=\left\{  \delta_{M}\mid M\in\mathcal{S}\right\}  $ and
for every $\delta\in D,$ define $W_{\delta}$ as the set of all $\mathcal{C}%
\left(  X\right)  $ such that $X\subseteq$ \textsf{fin }and there is
$M\in\mathcal{S}$ such that $\delta_{M}\leq\delta$ and $X\in M.$ By recursion
over $\delta\in D,$ we will define $\mathcal{U}_{\delta},$ $\mathcal{A}%
_{\delta}$ and $\mathcal{P}_{\delta}$ such that:

\begin{enumerate}
\item $\mathcal{U}_{\delta}\subseteq\left[  \omega\right]  ^{\omega}$ while
$\mathcal{P}_{\delta}$ and $\mathcal{A}_{\delta}$ are families of compact
subsets of $\mathcal{P}\left(  \omega\right)  .$

\item $\mathcal{U}_{\gamma}\subseteq\mathcal{U}_{\delta}$ and $\mathcal{A}%
_{\gamma}\subseteq\mathcal{A}_{\delta}$ for every $\gamma\in D\cap\delta.$

\item $\mathcal{A}_{\delta}\subseteq W_{\delta}.$

\item $\mathcal{P}_{\delta}$ is the collection of all $\mathcal{C(}%
F^{f_{\delta}})$ such that $F:\omega\longrightarrow\mathcal{P(}$%
\textsf{fin}$)$\textsf{ is }$\subseteq$-decreasing and there is $M\in
\mathcal{S}$ with the property that $\delta_{M}=\delta,$ $F\in M$ and
$\mathcal{C(}F\left(  n\right)  )\in\mathcal{A}_{<\delta}$ for every
$n\in\omega$ (where $\mathcal{A}_{<\delta}=%
{\textstyle\bigcup\limits_{\gamma\in D\cap\delta}}
\mathcal{A}_{\xi}$).

\item $\mathcal{U}_{\delta}=%
{\textstyle\bigcup}
\left\{  \mathcal{P}_{\gamma}\mid\gamma\in D\cap\left(  \delta+1\right)
\right\}  .$

\item $\mathcal{U}_{\delta}\cup%
{\textstyle\bigcup}
\mathcal{A}_{\delta}$ is centered.

\item $\mathcal{A}_{\delta}$ is maximal with respect to points 3 and 6.
\end{enumerate}

\qquad\ \ \ \ \qquad\ \ \ 

Assume we are at step $\delta\in D$ and $\mathcal{U}_{\gamma},$ $\mathcal{A}%
_{\gamma}$ and $\mathcal{P}_{\gamma}$ have been defined for all $\gamma\in
D\cap\delta.$ In case $\delta$ is the minimum of $D,$ we have $\mathcal{U}%
_{\delta}=\mathcal{P}_{\delta}=\emptyset.$ Choose $\mathcal{A}_{\delta
}\subseteq W_{\delta}$ any maximal compatible set of compact sets such that $%
{\textstyle\bigcup}
\mathcal{A}_{\delta}$ extends the Fr\'{e}chet filter. Now consider the case
where $\delta$ is not the least member of $D.$ Note that $\mathcal{U}_{\delta
}$ and $\mathcal{P}_{\delta}$ are defined from $\mathcal{A}_{<\delta},$ so we
only need to find $\mathcal{A}_{\delta}.$ Define $\mathcal{U}_{<\delta}=%
{\textstyle\bigcup\limits_{\xi\in D\cap\delta}}
\mathcal{U}_{\xi}$ and note that $\mathcal{U}_{<\delta}\cup%
{\textstyle\bigcup}
\mathcal{A}_{<\delta}$ is centered, since (by the recursion hypothesis) it is
an increasing union of centered sets. We now prove the following:

\begin{claim}
$\mathcal{U}_{\delta}\cup%
{\textstyle\bigcup}
\mathcal{A}_{<\delta}$ is centered.
\end{claim}

\qquad\ \ \ 

Let $B_{0},...,B_{n}\in\mathcal{U}_{\delta}\cup%
{\textstyle\bigcup}
\mathcal{A}_{<\delta}$, for every $i\leq n,$ we find $M_{i}\in\mathcal{S}$ and
$F_{i}\in M_{i}$ in the following way:

\begin{enumerate}
\item If $B_{i}\in%
{\textstyle\bigcup}
\mathcal{A}_{<\delta},$ let $M_{i}$ such that $\delta_{i}=\delta_{M_{i}%
}<\delta$ and there is $X_{i}$\textsf{ }$\in M_{i}$ such that $B_{i}%
\in\mathcal{C}\left(  X_{i}\right)  $ and $\mathcal{C}\left(  X_{i}\right)
\in\mathcal{A}_{\delta_{i}}.$ Let $F_{i}:\omega\longrightarrow$ $\mathcal{P(}%
$\textsf{fin)} be the constant sequence with value $X_{i}.$

\item If $B_{i}\in\mathcal{U}_{\delta},$ let $M_{i}$ such that $\delta
_{i}=\delta_{M_{i}}\leq\delta$ and $F_{i}:\omega\longrightarrow$
$\mathcal{P(}$\textsf{fin) }$\in M_{i}$ such that $B_{i}\in\mathcal{C(}%
F_{i}^{f_{\delta_{i}}})$ and each $\mathcal{C}\left(  F_{i}\left(  k\right)
\right)  $ is in $\mathcal{A}_{<\delta_{i}}.$
\end{enumerate}

\qquad\ \ 

It might be possible that for some $i\leq n$ both clauses apply. If that is
the case, we can choose to follow either one of them. For each $i\leq n,$ we
have the following:

\begin{enumerate}
\item $B_{i}\in\mathcal{C(}F_{i}^{f_{\delta_{i}}}).$

\item $F_{i}:\omega\longrightarrow$ $\mathcal{P(}$\textsf{fin) }$\in M_{i}$
and is $\subseteq$-decreasing$.$

\item $\mathcal{C}\left(  F_{i}\left(  k\right)  \right)  \in\mathcal{A}%
_{<\delta}$ for all $k\in\omega.$
\end{enumerate}

\qquad\ \ 

By taking a reenumeration and possibly picking more elements of $\mathcal{U}%
_{\delta}\cup%
{\textstyle\bigcup}
\mathcal{A}_{<\delta},$ we may assume that $\left\langle M_{0},...,M_{n}%
\right\rangle $ is $\delta$-increasing and $\delta_{n}=\delta.$ Let $m\leq n$
be the least such that $\delta_{m}=\delta.$ We claim that $H=\mathcal{C(}%
F_{0}^{f_{\delta_{0}}})\cup...\cup\mathcal{C(}F_{m-1}^{f_{\delta_{m-1}}})\cup%
{\textstyle\bigcup}
\left\{  \mathcal{C}\left(  F_{i}\left(  k\right)  \right)  \mid m\leq i\leq
n\wedge k\in\omega\right\}  $ is contained in $\mathcal{U}_{<\delta}\cup%
{\textstyle\bigcup}
\mathcal{A}_{<\delta}.$ Pick $i\leq n.$ We have the following cases:

\begin{enumerate}
\item If $B_{i}\in%
{\textstyle\bigcup}
\mathcal{A}_{<\delta},$ we have that $\mathcal{C}\left(  X_{i}\right)
\in\mathcal{A}_{<\delta}$ and $F_{i}^{\delta_{i}}=X_{i}$.

\item If $B_{i}\in\mathcal{U}_{\delta}$ and $i<m,$ then $F_{i}:\omega
\longrightarrow\mathcal{A}_{<\delta_{i}}$ which implies that $F_{i}%
^{f_{\delta_{i}}}\in\mathcal{U}_{\delta_{i}}\subseteq\mathcal{U}_{\delta}.$

\item If $B_{i}\in\mathcal{U}_{\delta}$ and $m\leq i,$ there is nothing to do
since we already noted that $\mathcal{C}\left(  F_{i}\left(  k\right)
\right)  \in\mathcal{A}_{<\delta}$ for all $k\in\omega.$
\end{enumerate}

\qquad\ \ \ \ 

Recall that $\mathcal{U}_{<\delta}\cup%
{\textstyle\bigcup}
\mathcal{A}_{<\delta}$ is centered, so $H$ is centered as well. We are now in
position to invoke Lemma \ref{Lema pathway strong ppoint} and conclude that $%
{\textstyle\bigcup\limits_{i\leq n}}
\mathcal{C(}F_{i}^{f_{\delta_{i}}})$ is centered. Since $B_{i}\in
\mathcal{C(}F_{i}^{f_{\delta_{i}}}),$ this finishes the proof of the claim. We
use Zorn's Lemma and find $\mathcal{A}_{\delta}\subseteq W_{\delta}$ extending
$\mathcal{A}_{<\delta}$ such that $\mathcal{U}_{\delta}\cup%
{\textstyle\bigcup}
\mathcal{A}_{\delta}$ is centered and it is maximal with these properties.

\qquad\qquad\qquad

After completing the recursion, define $\mathcal{U=}%
{\textstyle\bigcup\limits_{\delta\in D}}
\mathcal{U}_{\delta}$ and $\mathcal{A=}%
{\textstyle\bigcup\limits_{\delta\in D}}
\mathcal{A}_{\delta}.$ We will now prove the following:

\begin{claim}
\qquad\ \ \ 

\begin{enumerate}
\item $\mathcal{U\cup}%
{\textstyle\bigcup}
\mathcal{A}$ is centered.

\item $\mathcal{U}=%
{\textstyle\bigcup}
\mathcal{A}.$

\item $\mathcal{U}$ is an ultrafilter.

\item If $X\in\left(  \mathcal{U}^{<\omega}\right)  ^{+},$ then $\mathcal{C}%
\left(  X\right)  \in\mathcal{A}.$

\item $\mathcal{U}$ is a strong \textsf{P}-point.
\end{enumerate}
\end{claim}

\qquad\qquad\ \ \ \ 

Since $\mathcal{U\cup}%
{\textstyle\bigcup}
\mathcal{A}$ is equal to the union of $\mathcal{U}_{\delta}\mathcal{\cup}%
{\textstyle\bigcup}
\mathcal{A}_{\delta}$ and these are increasing and centered, it follows that
$\mathcal{U\cup}%
{\textstyle\bigcup}
\mathcal{A}$ is centered. We prove that $\mathcal{U\subseteq}$ $%
{\textstyle\bigcup}
\mathcal{A}.$ Let $B\in\mathcal{U},$ $\delta\in D$ and $M\in\mathcal{S}$ such
that $\delta_{M}=\delta$ and $B\in M.$ Let $X=\left[  B\right]  ^{1}$ and
recall that $\mathcal{C}\left(  X\right)  =\left\{  A\subseteq\omega\mid
B\subseteq A\right\}  $ and $\mathcal{C}\left(  X\right)  \in W_{\delta}$.
Clearly $\mathcal{U}_{\delta}\cup\mathcal{A}_{\delta}\cup\mathcal{C}\left(
X\right)  $ is centered, so by the maximality of $\mathcal{A}_{\delta},$ we
get that $\mathcal{C}\left(  X\right)  \in\mathcal{A}_{\delta},$ hence $B\in%
{\textstyle\bigcup}
\mathcal{A}.$ Now, take $A\in%
{\textstyle\bigcup}
\mathcal{A}.$ Find $\delta\in D$ and $Y\in W_{\delta}$ such that
$A\in\mathcal{C}\left(  Y\right)  .$ Let $F$ be the constant sequence with
value $Y.$ We now find $\gamma\in D$ and $M\in\mathcal{S}$ such that
$\delta_{M}=\gamma$ and $\delta,F\in M.$ We have that $\mathcal{C}\left(
Y\right)  =\mathcal{C}\left(  F^{f_{\gamma}}\right)  \subseteq\mathcal{U}%
_{\gamma}.$

\qquad\qquad\ \ \qquad\ \ \ 

We now prove that $\mathcal{U}$ is an ultrafilter. Let $A,B,C,E\subseteq
\omega$ such that $A,B\in\mathcal{U}$ and $A\subseteq C.$ Find $\delta\in D$
and $M\in\mathcal{S}$ such that $\delta_{M}=\delta$ and $A,B,C,E\in M.$ Let
$X=\left[  A\cap B\right]  ^{1},$ which is in $W_{\delta}.$ Since $\mathcal{U=%
{\textstyle\bigcup}
A}$ is centered, it follows that $\mathcal{U}_{\delta}\cup\mathcal{C}\left(
X\right)  \cup%
{\textstyle\bigcup}
\mathcal{A}_{\delta}$ is centered. By the maximality of $\mathcal{A}_{\delta
},$ we get that $\mathcal{C}\left(  X\right)  \in\mathcal{A}_{\delta}.$ Since
$A\cap B,$ $C\in$ $\mathcal{C}\left(  X\right)  ,$ it follows that $A\cap
B,C\in\mathcal{U}.$ Moreover, we know that either $\mathcal{U}_{\delta}\cup%
{\textstyle\bigcup}
\mathcal{A}_{\delta}\cup\mathcal{C}(\left[  E\right]  ^{1})$ is centered or
$\mathcal{U}_{\delta}\cup%
{\textstyle\bigcup}
\mathcal{A}_{\delta}\cup\mathcal{C(}\left[  \omega\setminus E\right]  ^{1})$
is centered, so by the maximality of $\mathcal{A}_{\delta},$ we obtain that
$\mathcal{C}\left(  \left[  E\right]  ^{1}\right)  \in\mathcal{A}_{\delta}$ or
$\mathcal{C}\left(  \left[  \omega\setminus E\right]  ^{1}\right)
\in\mathcal{A}_{\delta}.$ Finally, recall that in the first step of the
recursion we made sure that $\mathcal{U}$ contains the Fr\'{e}chet filter.

\qquad\qquad\qquad\ \ 

Let $X\in\left(  \mathcal{U}^{<\omega}\right)  ^{+}$. Since $\mathcal{U}$ is
an ultrafilter, we know that $\mathcal{C}\left(  X\right)  \subseteq
\mathcal{U}$ by Lemma \ref{Lema positivo ultrafiltro}. Find $\delta\in D$ and
$M\in\mathcal{S}$ such that $\delta_{M}=\delta$ and $X\in M.$ Clearly
$\mathcal{U}_{\delta}\cup\mathcal{C}\left(  X\right)  \cup%
{\textstyle\bigcup}
\mathcal{A}_{\delta}$ is centered, so by the maximality of $\mathcal{A}%
_{\delta}$ we conclude that $\mathcal{C}\left(  X\right)  \in\mathcal{A}%
_{\delta}.$

\qquad\qquad

It only remains to prove that $\mathcal{U}$ is a strong \textsf{P}-point. We
use Theorem \ref{Equivalencias strong ppoint}. Let $F:\omega\longrightarrow
\left(  \mathcal{U}^{<\omega}\right)  ^{+}$ be $\subseteq$-decreasing$.$ By
the previous point of the claim, for every $n\in\omega,$ we can find
$\delta_{n}\in D$ such that $\mathcal{C}\left(  F\left(  n\right)  \right)
\in\mathcal{A}_{\delta_{n}}.$ We now choose $\delta\in D$ and $M\in
\mathcal{S}$ with $\delta_{M}=\delta$ such that $F\in M$ and $\delta
_{n}<\delta$ for every $n\in\omega.$ In this way, $\mathcal{C(}F^{f_{\delta}%
})\in\mathcal{P}_{\delta}$ and then it is contained in $\mathcal{U}.$ We
conclude that $F^{f_{\delta}}\in\left(  \mathcal{U}^{<\omega}\right)  ^{+}$ by
applying Lemma \ref{Lema positivo ultrafiltro} once again.
\end{proof}

\section{A Gruff ultrafilter from a multiple $\mathfrak{d}$%
-pathway\label{Gruff}}

We now turn our attention to ultrafilters on the rational
numbers\footnote{When discussing the rational numbers, we always assume it is
equipped with its usual topology.}, which we denote by $\mathbb{Q}$. A
particularly nice combinatorial feature of the rational numbers (which is not
true for the real numbers), is that it has no \emph{Bernstein subsets}. In
other words, if we split the rational numbers into two pieces, then at least
one of them contains a perfect subset\footnote{Recall that perfect sets are
non-empty.}. This property motivates the following definition:

\begin{definition}
Let $\mathcal{U}$ be an ultrafilter on $\mathbb{Q}.$ We say that $\mathcal{U}$
is a \emph{Gruff ultrafilter }if it has a base of perfect sets.
\end{definition}

\qquad\ \ \ 

Gruff ultrafilters were introduced by van Douwen in \cite{gruffvanDouwen}
while studying $\beta\mathbb{Q}$ (the \v{C}ech-Stone compactification of
$\mathbb{Q})$ who asked if there are such ultrafilters. He was able to to
prove they exist in case that \textsf{cov}$\left(  \mathcal{M}\right)
=\mathfrak{c}$ and later Copl\'{a}kov\'{a} and Hart in \cite{Hart} obtained
the same conclusion under $\mathfrak{b=c}.$ Both of these results were
improved by Fern\'{a}ndez-Bret\'{o}n and Hru\v{s}\'{a}k in
\cite{GruffUltrafilters} where a Gruff ultrafilter is obtained from
$\mathfrak{d=c}.$ Fern\'{a}ndez-Bret\'{o}n and Hru\v{s}\'{a}k were also
interested in the existence of Gruff ultrafilters in the random model and
constructed one using a pathway (see also \cite{GeneralizedPathways}).
Unfortunately, as discussed before, it is not known if pathways exist in the
random model. We will now build a Gruff ultrafilter from a multiple
$\mathfrak{d}$-pathway. Our approach takes inspiration from
\cite{GruffUltrafilters} (\cite{GeneralizedPathways}) and
\cite{PFiltersCohenRandomLaver}. Apart from the papers already mentioned, the
reader may consult \cite{GruffEmmanuel}, \cite{CPAbook} and \cite{GruffCPA} to
learn more about Gruff ultrafilters and \cite{irr} for more on combinatorics
of scattered subsets of the rationals.

\qquad\ \ \qquad\ \ 

Denote by \textsf{scatt }the ideal of scattered subsets of $\mathbb{Q}$ and by
\textsf{bscatt }the ideal generated by both the scattered sets and the bounded
(from above) subsets of $\mathbb{Q}.$ When constructing a Gruff ultrafilter,
it is sometimes more convenient to build one which has a base of perfect
unbounded subsets (see \cite{GruffUltrafilters}). Given $A\subseteq
\mathbb{Q},$ the \emph{crowded kernel of} $A$ (denoted by $K\left(  A\right)
$) is the union of all the crowded subsets of $A.$ The following are simple
remarks regarding this notion:

\begin{lemma}
Let $A\subseteq\mathbb{Q}.$ \label{prop basicas scatt}

\begin{enumerate}
\item If $A\in$ \textsf{scatt}$^{+}$, then $K\left(  A\right)  $ is crowded.

\item $K\left(  A\right)  $ is the largest crowded subset contained in $A.$

\item If $A\in$ \textsf{bscatt}$^{+}$, then $K\left(  A\right)  $ is crowded
and unbounded.

\item The symmetric difference between $A$ and $K\left(  A\right)  $ is in
\textsf{scatt.}

\item If $\mathcal{F}$ is a filter on $\mathbb{Q}$ such that \textsf{scatt}%
$^{\ast}\subseteq\mathcal{F}$ and $A\in\mathcal{F},$ then $K\left(  A\right)
\in\mathcal{F}.$
\end{enumerate}
\end{lemma}

\qquad\ \ 

We will now recall a very useful notion from \cite{GruffUltrafilters}. From
now on, fix an enumeration $\mathbb{Q}=\left\{  q_{n}\mid n\in\omega\right\}
$. For $q\in\mathbb{Q}$ and $r>0,$ we denote by $B\left(  q,r\right)  $ the
open ball of $q$ with radius $r.$ Given a function $f\in\omega^{\omega}$ and
$n\in\omega,$ denote $J_{f}\left(  n\right)  =B(q_{n},\frac{\sqrt{2}}{k}),$
where $k$ is the least natural number such that $q_{m}\notin B(q_{n}%
,\frac{\sqrt{2}}{k})$ for every $m\leq f\left(  n\right)  $ such that $m\neq
n$ (the purpose of $\sqrt{2}$ is only to ensure that $J_{f}\left(  n\right)  $
is a clopen subset of $\mathbb{Q},$ evidently we can use any other positive
irrational number). Intuitively, we are making $J_{f}\left(  n\right)  $ as
large as possible with the restriction that it can not include any $q_{m}$ for
which $m\leq f\left(  n\right)  $ and $m\neq n.$

\begin{definition}
Let $X\subseteq\mathbb{Q}$ and $f\in\omega^{\omega}.$ Define $X\left(
f\right)  =\mathbb{Q\setminus}%
{\textstyle\bigcup\limits_{n\notin X}}
J_{f}\left(  n\right)  .$
\end{definition}

\qquad\qquad\ \ 

The following two results can be found in \cite{GruffUltrafilters}:

\begin{lemma}
Let $X,Y\subseteq\mathbb{Q}$ and $f,g\in\omega^{\omega}.$

\begin{enumerate}
\item $X\left(  f\right)  $ is a closed subset of $X.$

\item If $X\subseteq Y,$ then $X\left(  f\right)  \subseteq Y\left(  f\right)
.$

\item $X\left(  f\right)  \cap Y\left(  f\right)  =\left(  X\cap Y\right)
\left(  f\right)  .$

\item If $f\leq g,$ then $X\left(  f\right)  \subseteq X\left(  g\right)  .$
\end{enumerate}
\end{lemma}

\begin{proposition}
Let $X\subseteq\mathbb{Q}$ be crowded and unbounded. There is $h\in
\omega^{\omega}$ such that for every increasing $g\in\omega^{\omega},$ if
$g\nleq^{\ast}h,$ then $X\left(  g\right)  $ is perfect and unbounded.
\label{Prop David Michael}
\end{proposition}

\qquad\qquad\ \ \ \ \ \qquad\qquad\ \ \ 

Using Proposition \ref{Prop David Michael}, for every $X\subseteq\mathbb{Q}$
that is crowded and unbounded, we fix a function $h_{X}\in\omega^{\omega}$
such that for every increasing $g\in\omega^{\omega},$ if $g\nleq^{\ast}h_{X},$
then $X\left(  g\right)  $ is perfect and unbounded.

\qquad\qquad\ \ \ 

For a rational number $q\in\mathbb{Q},$ define $\left\langle q\right\rangle
_{0}=\left\{  A\subseteq\mathbb{Q}\mid q\notin A\right\}  $ and $\left\langle
q\right\rangle _{1}=\left\{  A\subseteq\mathbb{Q}\mid q\in A\right\}  .$ We
endow $\mathcal{P}\left(  \mathbb{Q}\right)  $ with the topology that has as a
subbase the family $\left\{  \left\langle q\right\rangle _{i}\mid
q\in\mathbb{Q}\wedge i<2\right\}  .$ Evidently, $\mathcal{P}\left(
\mathbb{Q}\right)  $ with this topology is homeomorphic to $\mathcal{P}\left(
\omega\right)  .$ The following lemma refers to this topology.

\begin{lemma}
\label{Complejidad Gruff}\qquad\ \ \qquad\ \ 

\begin{enumerate}
\item The collection of all crowded unbounded subsets of $\mathbb{Q}$ is
$G_{\delta}.$

\item The ideal \textsf{bscatt }is coanalytic.
\end{enumerate}
\end{lemma}

We now prove the following:

\begin{lemma}
Let $\left(  \mathcal{B},\mathcal{S}\right)  $ be a multiple $\mathfrak{d}%
$-pathway and $\left\langle M_{0},...,M_{n}\right\rangle $ a $\delta
$-increasing sequence of models from $\mathcal{S}.$ Let $m\leq n$ be the first
one such that $\delta_{m}=\delta_{n}.$ For every $i\leq n,$ pick $X_{i}\in
M_{i}\cap$ \textsf{bscatt}$^{+}.$

\qquad\qquad\qquad\ \ \ \qquad\ \ \ \ \qquad\ \ \ \ \ 

If $X_{0}\left(  f_{\delta_{0}}\right)  \cap...\cap X_{m-1}\left(
f_{\delta_{m-1}}\right)  \cap X_{m}\cap...\cap X_{n}\in$ \textsf{bscatt}%
$^{+},$ then $%
{\textstyle\bigcap\limits_{i\leq n}}
X_{i}\left(  f_{\delta_{i}}\right)  \in$ \textsf{bscatt}$^{+}.$
\label{lemma pathway Gruff}
\end{lemma}

\begin{proof}
Define the relation $R\subseteq\mathcal{P}\left(  \mathbb{Q}\right)
^{n+1}\times\omega^{\omega}$ where $R\left(  Y_{0},...,Y_{n},f\right)  $ holds
in case one of the following conditions is met:

\begin{enumerate}
\item $Y_{0}\left(  f_{\delta_{0}}\right)  \cap...\cap Y_{m-1}\left(
f_{\delta_{m-1}}\right)  \cap Y_{m}\cap...\cap Y_{n}\in$ \textsf{bscatt.}

\item $Y_{0}\left(  f_{\delta_{0}}\right)  \cap...\cap Y_{m-1}\left(
f_{\delta_{m-1}}\right)  \cap Y_{m}\left(  f\right)  \cap...\cap Y_{n}\left(
f\right)  \in$ \textsf{bscatt}$^{+}$\textsf{.}
\end{enumerate}

\qquad\ \ 

Since $f_{\delta_{0}},...,f_{\delta_{m-1}}\in M_{n},$ we conclude that $R\in
M_{n}.$ By Lemma \ref{Complejidad Gruff}, the first clause is coanalytic and
the second one is analytic$,$ so $R$ is projective. We now prove that it is
$\leq^{\ast}$-adequate. Let $Y_{0},...,Y_{n}\subseteq\mathbb{Q},$ if
$Z=Y_{0}\left(  f_{\delta_{0}}\right)  \cap...\cap Y_{m-1}\left(
f_{\delta_{m-1}}\right)  \cap Y_{m}\cap...\cap Y_{n}\in$ \textsf{bscatt},
there is nothing to do, so assume otherwise. We claim that $h_{K\left(
Z\right)  }$ is an $R$-control for $\left(  Y_{0},...,Y_{n}\right)  .$ To see
this, pick $g\in\omega^{\omega}$ increasing such that $g\nleq^{\ast
}h_{K\left(  Z\right)  }.$ We now have the following:\bigskip

$\hfill%
\begin{array}
[c]{lll}%
K\left(  Z\right)  \left(  g\right)  & \subseteq & Z\left(  g\right) \\
& = &
{\textstyle\bigcap\limits_{i<m}}
\left(  Y_{i}\left(  f_{\delta_{i}}\right)  \right)  \left(  g\right)  \cap%
{\textstyle\bigcap\limits_{m\leq i\leq n}}
Y_{i}\left(  g\right) \\
& \subseteq &
{\textstyle\bigcap\limits_{i<m}}
\left(  Y_{i}\left(  f_{\delta_{i}}\right)  \right)  \cap%
{\textstyle\bigcap\limits_{m\leq i\leq n}}
Y_{i}\left(  g\right)
\end{array}
\qquad\ \hfill$

\qquad\ \ \ \ \qquad\ \ \ \ \ \qquad\ \ \ \bigskip

By Proposition \ref{Prop David Michael}, we know that $K\left(  Z\right)
\left(  g\right)  $ is perfect and unbounded, so $%
{\textstyle\bigcap\limits_{i<m}}
\left(  Y_{i}\left(  f_{\delta_{i}}\right)  \right)  \cap%
{\textstyle\bigcap\limits_{m\leq i\leq n}}
Y_{i}\left(  g\right)  $ is not in \textsf{bscatt. }The conclusion of the
lemma follows since $\left(  \mathcal{B},\mathcal{S}\right)  $ is a multiple
$\mathfrak{d}$-pathway.
\end{proof}

\qquad\ \ \ \qquad\ \ 

We now proceed to prove the main result of the section.

\begin{theorem}
If there is a multiple $\mathfrak{d}$-pathway, then there is a Gruff ultrafilter.
\end{theorem}

\begin{proof}
Let $\left(  \mathcal{B},\mathcal{S}\right)  $ be a multiple $\mathfrak{d}%
$-pathway. Define $D=\left\{  \delta_{M}\mid M\in\mathcal{S}\right\}  $ and
for $\delta\in D,$ denote $W_{\delta}=%
{\textstyle\bigcup\limits_{\delta_{M}\leq\delta}}
M\cap\mathcal{P}\left(  \mathbb{Q}\right)  .$ By recursion over $\delta\in D,$
we shall find $\mathcal{U}_{\delta},\mathcal{A}_{\delta}$ and $\mathcal{P}%
_{\delta}$ with the following properties:

\begin{enumerate}
\item $\mathcal{U}_{\delta}$ and $\mathcal{P}_{\delta}$ are families of
perfect and unbounded sets.

\item $\mathcal{A}_{\delta}\subseteq W_{\delta}\cap$ \textsf{bscatt}$^{+}.$

\item If $\gamma\in D\cap\delta,$ then $\mathcal{A}_{\gamma}\subseteq
\mathcal{A}_{\delta}$ and $\mathcal{U}_{\gamma}\subseteq\mathcal{U}_{\delta}.$

\item $\mathcal{P}_{\delta}$ is the family of all $X\left(  f_{\delta}\right)
$ for which there is $M\in\mathcal{S}$ for which $\delta_{M}=\delta$ and $X\in
M\cap\mathcal{A}_{<\delta}$ is crowded (where $\mathcal{A}_{<\delta}=%
{\textstyle\bigcup\limits_{\gamma\in D\cap\delta}}
\mathcal{A}_{\gamma}$).

\item $\mathcal{U}_{\delta}=%
{\textstyle\bigcup}
\left\{  \mathcal{P}_{\xi}\mid\xi\in D\cap\left(  \delta+1\right)  \right\}
.$

\item $\mathcal{U}_{\delta}\cup\mathcal{A}_{\delta}$ generates a filter
contained in \textsf{bscatt}$^{+}.$

\item $\mathcal{A}_{\delta}$ is maximal with respect to points 2 and 6.
\end{enumerate}

\qquad\ \ \ \qquad\ \ 

Before starting the construction, note that $\mathcal{A}_{\delta}$ will have
the following property: If $B\in\mathcal{A}_{\delta},$ then $K\left(
B\right)  \in\mathcal{A}_{\delta}.$ To see this, let $M\in\mathcal{S}$ such
that $B\in M$ and $\delta_{M}\leq\delta.$ Since $B\in M,$ we get that
$K\left(  B\right)  \in M,$ hence $K\left(  B\right)  \in W_{\delta}.$ Call
$\mathcal{F}$ the filter generated by $\mathcal{U}_{\delta}\cup\mathcal{A}%
_{\delta}\cup$ \textsf{bscatt}$^{\ast}$. Since $B\in\mathcal{F},$ then
$K\left(  B\right)  \in\mathcal{F}$ (see Lemma \ref{prop basicas scatt}) which
implies that $\mathcal{U}_{\delta}\cup\mathcal{A}_{\delta}\cup\left\{
K\left(  B\right)  \right\}  $ generates a filter contained in \textsf{bscatt}%
$^{+}.$ By the maximality of $\mathcal{A}_{\delta},$ we conclude that
$K\left(  B\right)  \in\mathcal{A}_{\delta}.$

\ \ \qquad\ \qquad\qquad\qquad\ \ \ \ \qquad\qquad\ \ \ 

Assume we are at step $\delta\in D$ and $\mathcal{U}_{\gamma},$ $\mathcal{A}%
_{\gamma}$ and $\mathcal{P}_{\gamma}$ have been defined for all $\gamma\in
D\cap\delta.$ In case $\delta$ is the minimum of $D,$ we have $\mathcal{U}%
_{\delta}=\mathcal{P}_{\delta}=\emptyset.$ Choose $\mathcal{A}_{\delta
}\subseteq W_{\delta}\cap$ \textsf{scatt}$^{+}$ any maximal centered set
extending the filter of cobounded subsets of $\mathbb{Q}.$ Now consider the
case where $\delta$ is not the least member of $D.$ Note that $\mathcal{U}%
_{\delta}$ and $\mathcal{P}_{\delta}$ are defined from $\mathcal{A}_{<\delta
},$ so we only need to find $\mathcal{A}_{\delta},$ but first we need to prove
that both $\mathcal{U}_{\delta}$ and $\mathcal{P}_{\delta}$ consist of perfect
and unbounded sets. It is enough to prove it for $\mathcal{P}_{\delta}.$ Let
$M\in\mathcal{S}$ with $\delta_{M}=\delta$ and $X\in M\cap\mathcal{A}%
_{<\delta}$ is crowded. Moreover, $X$ is also unbounded since $\mathcal{A}%
_{<\delta}$ extends the filter of cobounded sets. We need to prove that
$X\left(  f_{\delta}\right)  $ is perfect and unbounded. Since $X\in M,$ it
follows that $h_{K}$ is also in $M.$ Since $f_{\delta}$ is unbounded over $M,$
we get that $X\left(  f_{\delta}\right)  $ is perfect and unbounded by
Proposition \ref{Prop David Michael}.

\qquad\ \ \ 

Define $\mathcal{U}_{<\delta}=%
{\textstyle\bigcup\limits_{\xi\in D\cap\delta}}
\mathcal{U}_{\xi}$ and note that $\mathcal{U}_{<\delta}\cup\mathcal{A}%
_{<\delta}$ generates a filter contained in \textsf{bscatt}$^{+}$ by the
recursion hypothesis. We now prove the following:

\begin{claim}
$\mathcal{U}_{\delta}\cup\mathcal{A}_{<\delta}$ generates a filter contained
in \textsf{bscatt}$^{+}.$
\end{claim}

\qquad\ \ \ 

Let $B_{0},...,B_{n}\in\mathcal{U}_{\delta}\cup\mathcal{A}_{<\delta}$, for
each $i\leq n,$ we find $M_{i}\in\mathcal{S}$ and $X_{i}\in M_{i}$ in the
following way:

\begin{enumerate}
\item In case $B_{i}\in\mathcal{A}_{<\delta},$ choose $M_{i}$ for which
$\delta_{i}=\delta_{M_{i}}<\delta$ and $B_{i}\in M_{i}.$ Let $X_{i}=K\left(
B_{i}\right)  .$

\item If $B_{i}\in\mathcal{U}_{\delta},$ choose $M_{i}$ with $\delta
_{i}=\delta_{M_{i}}\leq\delta$ and $X_{i}\in M_{i}\cap\mathcal{A}_{<\delta
_{i}}$ crowded such that $B_{i}=X_{i}\left(  f_{\delta_{i}}\right)  .$
\end{enumerate}

\qquad\ \ 

It might be possible that for some $i\leq n$ both clauses apply. If that is
the case, we use either of them. For each $i\leq n,$ we have the following:

\begin{enumerate}
\item $X_{i}\in M_{i}\cap\mathcal{A}_{<\delta}$ and is both perfect and unbounded.

\item $X_{i}\left(  f_{\delta_{i}}\right)  \subseteq B_{i}.$
\end{enumerate}

\qquad\ \ 

By taking a reenumeration and possibly picking more elements of $\mathcal{U}%
_{\delta}\cup\mathcal{A}_{<\delta},$ we may assume that $\left\langle
M_{0},...,M_{n}\right\rangle $ is $\delta$-increasing and $\delta_{n}=\delta.$
Let $m\leq n$ be the least such that $\delta_{m}=\delta.$ We claim that
$X_{0}\left(  f_{\delta_{0}}\right)  ,...,X_{m-1}\left(  f_{\delta_{m-1}%
}\right)  ,X_{m},...,X_{n}\in$ $\mathcal{U}_{<\delta}\cup\mathcal{A}_{<\delta
}.$ Pick $i\leq n.$ We have the following cases:

\begin{enumerate}
\item If $B_{i}\in\mathcal{A}_{<\delta},$ then $X_{i}=K\left(  B_{i}\right)
\in\mathcal{A}_{\delta_{i}},$ so $X_{i}\left(  f_{\delta_{i}}\right)
\in\mathcal{U}_{\delta_{i}}$.

\item If $B_{i}\in\mathcal{U}_{\delta}$ and $i<m,$ then $X_{i}\in
\mathcal{A}_{\delta_{i}},$ so $X_{i}\left(  f_{\delta_{i}}\right)
\in\mathcal{U}_{\delta_{i}}$.

\item If $B_{i}\in\mathcal{U}_{\delta}$ and $m\leq i,$ we already knew that
$X_{i}\in\mathcal{A}_{<\delta}.$
\end{enumerate}

\qquad\ \ \ \qquad\ \ \ \ \ 

Recall that $\mathcal{U}_{<\delta}\cup\mathcal{A}_{<\delta}$ generates a
filter contained in \textsf{bscatt}$^{+}$, so $X_{0}\left(  f_{\delta_{0}%
}\right)  \cap...\cap X_{m-1}\left(  f_{\delta_{m-1}}\right)  \cap X_{m}%
\cap...\cap X_{n}\in$ \textsf{bscatt}$^{+}$. We are now in position to call
Lemma \ref{lemma pathway Gruff} and conclude that $%
{\textstyle\bigcap\limits_{i\leq n}}
X_{i}\left(  f_{\delta_{i}}\right)  \in$ \textsf{bscatt}$^{+}$. Since
$X_{i}\left(  f_{\delta_{i}}\right)  \subseteq B_{i},$ this finishes the proof
of the claim. We now invoke Zorn's Lemma and find $\mathcal{A}_{\delta
}\subseteq W_{\delta}$ extending $\mathcal{A}_{<\delta}$\ as desired.

\qquad\qquad\qquad

After completing the recursion, define $\mathcal{A=}%
{\textstyle\bigcup\limits_{\delta\in D}}
\mathcal{A}_{\delta}\ $and$\ \mathcal{U}$ as the set of all $B\subseteq
\mathbb{Q}$ for which there is $U\in%
{\textstyle\bigcup\limits_{\delta\in D}}
\mathcal{U}_{\delta}$ for which $B\subseteq U.$ We will now prove the following:

\begin{claim}
\qquad\ \ \ \qquad\ \ 

\begin{enumerate}
\item $\mathcal{U\cup A}$ generates a filter contained in \textsf{bscatt}%
$^{+}.$

\item $\mathcal{U=A}.$

\item $\mathcal{U}$ is an ultrafilter.

\item $\mathcal{U}$ is a Gruff ultrafilter.
\end{enumerate}
\end{claim}

\qquad\ \ \ \ 

The first point is easy, we now prove that $\mathcal{U=A}.$ We will first see
that $\mathcal{U\subseteq A}.$ Let $U\in\mathcal{U},$ find $\delta\in D$ and
$M\in\mathcal{S}$ such that there is $B\in\mathcal{U}_{\delta}$ for which
$B\subseteq U$ and $B,U\in M.$ It is clear that $\mathcal{U}_{\delta}%
\cup\mathcal{A}_{\delta}\cup\left\{  U\right\}  $ generates a filter contained
in \textsf{bscatt}$^{+}.$ Since $U\in W_{\delta},$ it follows by the
maximality of $\mathcal{A}_{\delta}$ that $U\in\mathcal{A}_{\delta}.$ We now
prove that $\mathcal{A\subseteq U}.$ Let $A\in\mathcal{A}_{\delta}$ for some
$\delta\in D.$ We now choose $\gamma\in D$ and $M\in\mathcal{S}$ such that
$\delta_{M}=\gamma$ and $A,\delta\in M.$ Since $A\in\mathcal{A}_{<\gamma},$ it
follows that $K=K\left(  A\right)  $ is also in $\mathcal{A}_{<\gamma}.$ In
this way, $K\left(  f_{\gamma}\right)  \in\mathcal{U}_{\gamma}$ and then
$A\in\mathcal{U}.$

\qquad\qquad\ \ 

The proof that $\mathcal{U}$ is an ultrafilter is similar to arguments used in
the proof of Theorems \ref{Teorema P punto} and \ref{Teorema Strong}. Finally,
it is Gruff since $%
{\textstyle\bigcup\limits_{\delta\in D}}
\mathcal{U}_{\delta}$ is a base of $\mathcal{U}$ consisting of perfect sets.
\end{proof}

\section{Combinatorics of elementary submodels \label{Seccion submodelos}}

Our current goal now is to prove that multiple $\mathfrak{d}$-pathways may
consistently exist. We will derive several combinatorial results concerning
countable elementary submodels which are the new insight for the main theorems
of the paper. The results in this section do not directly refer to pathways
and may be of independent interest.

\qquad\ \ \qquad\ \ \ 

Fix a regular cardinal $\kappa>\mathfrak{c}$ and $\trianglelefteq$ a well
order of \textsf{H}$\left(  \kappa\right)  .$ The following result is
well-known, we prove it for completeness.

\begin{lemma}
[\textsf{CH}]Assume that $M,N\in$ \textsf{Sub}$\left(  \kappa\right)  $ and
$\delta_{M}\leq\delta_{N}.$ \textsf{H}$\left(  \omega_{1}\right)  \cap
M\subseteq$ \textsf{H}$\left(  \omega_{1}\right)  \cap N.$
\label{Lema contencion Hw1}
\end{lemma}

\begin{proof}
Let $g:\omega_{1}\longrightarrow$ \textsf{H}$\left(  \omega_{1}\right)  $ be
the $\trianglelefteq$-minimal bijection, so it is in both $M$ and $N.$ Since
\textsf{H}$\left(  \omega_{1}\right)  \cap M=g\left[  \delta_{M}\right]  $ and
\textsf{H}$\left(  \omega_{1}\right)  \cap N=g\left[  \delta_{N}\right]  ,$
the result follows.
\end{proof}

\qquad\qquad\ \ 

We now extend the previous lemma:

\begin{lemma}
[\textsf{CH}]Let $M,N\in$ \textsf{Sub}$\left(  \kappa\right)  $ with
$\delta_{M}\leq\delta_{N}.$ If $A\in M\cap N$ and is a countable subset of the
ordinals, then $\mathcal{P}\left(  A\right)  \cap M\subseteq N.$
\label{Lema potencia contenida}
\end{lemma}

\begin{proof}
Let $B\in\mathcal{P}\left(  A\right)  \cap M$ and $\gamma=$ \textsf{OT}%
$\left(  A\right)  <\omega_{1}.$ Denote by $e:A\longrightarrow\gamma$ the
(unique) order isomorphism. Since $A\in M\cap N,$ it follows that $e\in M\cap
N.$ Clearly $e\left[  B\right]  \in$ \textsf{H}$\left(  \omega_{1}\right)
\cap M$, so $e\left[  B\right]  \in N$ by Lemma \ref{Lema contencion Hw1}.
Since $e^{-1}$ is also in $N,$ it follows that $B\in N.$
\end{proof}

\qquad\qquad\ \ 

If $A$ is a set of ordinals, we denote by $\overline{A}$ its closure in the
usual order topology. It is easy to see that the closure of a countable set is
also countable. In particular, if $M\in$ \textsf{Sub}$\left(  \kappa\right)  $
and $A\in M$ is a countable set of ordinals, then $\overline{A}\subseteq M.$
For us, \emph{a partition} $P$ is simply a collection of pairwise disjoint
sets ($\emptyset\in P$ is allowed) and a partition for a set $A$ is a
partition whose union is $A.$

\begin{lemma}
[\textsf{CH}]Let $M,N\in$\textsf{Sub}$\left(  \kappa\right)  $ with
$\delta_{M}\leq\delta_{N},$ $n\in\omega$ and $A,B\in\left[  \omega_{n}\right]
^{\leq\omega}$ such that $A\in M$ and $B\in N.$ \label{partir conjunto en dos}

\begin{enumerate}
\item There is a partition $P=\left\{  A_{0},A_{1}\right\}  \in M$ of $A$ such
that $A_{0}\in N$ and $A_{1}\cap B=\emptyset.$

\item $A\cap B\in N .$
\end{enumerate}
\end{lemma}

\begin{proof}
Note that the second point is a trivial consequence of the first. We prove the
first point by induction over $n.$ For $n\leq1,$ we have that $A\in N$ by
Lemma \ref{Lema contencion Hw1}, so we simply let $A_{0}=A\cap B$ and
$A_{1}=A\setminus B.$ Assume the lemma is true for $n,$ we prove that it is
true for $n+1$ as well. Denote $\beta=%
{\textstyle\bigcup}
\overline{A\cap B}+1$ and note that $\beta\in M\cap N.$ Fix $h:\beta
\longrightarrow\omega_{n}$ be the $\trianglelefteq$-least injective function,
clearly $h\in M\cap N.$ Define $C=h\left[  A\cap\beta\right]  $ and
$D=h\left[  B\cap\beta\right]  ,$ we have that $C\in M$ and $D\in N.$ We can
now apply the inductive hypothesis and find a partition $\left\{  C_{0}%
,C_{1}\right\}  \in M$ of $C$ such that $C_{0}\in N$ and $C_{1}\cap
D=\emptyset.$ Letting $A_{0}=h^{-1}\left(  C_{0}\right)  $ and $A_{1}%
=A\setminus h^{-1}\left(  C_{0}\right)  ,$ we have that $\left\{  A_{0}%
,A_{1}\right\}  \in M$ and $A_{0}\in N.$ We only need to prove that $A_{1}\cap
B=\emptyset.$ Assume that there is $\alpha\in A_{1}\cap B,$ it follows that
$\alpha<\beta$ and $h\left(  \alpha\right)  \in C\setminus C_{0}=C_{1}.$ In
this way, $h\left(  \alpha\right)  \in C_{1}\cap D,$ but this is a
contradiction since $C_{1}\cap D=\emptyset.$
\end{proof}

\qquad\ \ \qquad\ \ 

The following definition plays a similar role to the finite partitions used in
\cite{PFiltersCohenRandomLaver}.

\begin{definition}
Let $\left\langle M_{0},...,M_{n}\right\rangle $ be a $\delta$-increasing
sequence of models from \textsf{Sub}$\left(  \kappa\right)  $. We say that
$\mathcal{P}=\left\langle P_{i}\mid i\leq n\right\rangle $ i\emph{s a coherent
sequence of partitions for} $\left\langle M_{0},...,M_{n}\right\rangle $ if
for every $i\leq n,$ the following conditions hold:

\begin{enumerate}
\item $P_{i}\in M_{i}$ and is a finite partition consisting of countable
subsets of the ordinals.

\item If $i<j,$ then $P_{i}\cap M_{j}\subseteq P_{j}.$

\item $%
{\textstyle\bigcup\limits_{j\leq n}}
P_{j}$ is a partition.
\end{enumerate}
\end{definition}

\qquad\qquad\ \ 

The next lemma illustrates how to construct non-trivial coherent sequences of partitions.

\begin{lemma}
[\textsf{CH}]Let $\left\langle M_{0},...,M_{n}\right\rangle $ be a $\delta
$-increasing sequence of models from \textsf{Sub}$\left(  \kappa\right)  $,
$l\in\omega$ and $A_{i}\in M_{i}\cap\left[  \omega_{l}\right]  ^{\leq\omega}$
for every $i\leq n.$ There is $\mathcal{P}=\left\langle P_{i}\mid i\leq
n\right\rangle $ a coherent sequence of partitions for $\left\langle
M_{0},...,M_{n}\right\rangle $ such that $A_{i}\subseteq\cup P_{i}$ for every
$i\leq n.$ \label{Lema particion que cubre}
\end{lemma}

\begin{proof}
We proceed by induction over $n.$ For $n=0$ we simple take $P_{0}=\left\{
A_{0}\right\}  $ and we are done. Assume the lemma is true for $n,$ we will
see it is true for $n+1$ as well. Find $\left\langle P_{i}\mid i\leq
n\right\rangle $ a coherent sequence of partitions for $\left\langle
M_{0},...,M_{n}\right\rangle $ such that $A_{i}\subseteq\cup P_{i}$ for every
$i\leq n.$

\begin{claim}
There is $\left\langle R_{i}\mid i\leq n\right\rangle $ a coherent sequence of
partitions for $\left\langle M_{0},...,M_{n}\right\rangle $ with the following properties:

\begin{enumerate}
\item $\cup P_{j}\subseteq\cup R_{j}$ for every $j\leq n.$

\item For every $B\in%
{\textstyle\bigcup\limits_{i\leq n}}
R_{i},$ we have that $B\cap A_{n+1}\in M_{n+1}.$
\end{enumerate}
\end{claim}

\qquad\ \ 

Denote $P=%
{\textstyle\bigcup\limits_{i\leq n}}
P_{i}.$ Pick $C\in P$ and $i\leq n$ the minimal one for which $C\in P_{i}.$ We
can apply Lemma \ref{partir conjunto en dos} and find a partition $\left\{
C_{0},C_{1}\right\}  \in M_{i}$ of $C$ such that $C_{0}\in M_{n}$ and
$A_{n+1}\cap C_{1}=\emptyset.$ Define $R=\left\{  C_{u}\mid C\in P\wedge
u\in2\right\}  $ and for every $i\leq n,$ denote $R_{i}=R\cap M_{i}.$ We claim
that $\left\langle R_{i}\mid i\leq n\right\rangle $ is as desired. Clearly $R$
is a partition, if $B\in R$ then $B\cap A_{n+1}\in M_{n+1}$ and if $i<j,$ then
$R_{i}\cap M_{j}=R\cap M_{i}\cap M_{j}=R_{j}\cap M_{i}\subseteq R_{j}.$ It
remains to prove that $\cup P_{j}\subseteq\cup R_{j}$ for every $j\leq n.$ Let
$C\in P_{j}$ and find $i\leq j$ the first one for which $C\in P_{i}.$ Note
that $C\in M_{i}\cap M_{j}$ and $\delta_{i}\leq\delta_{j},$ so by Lemma
\ref{Lema potencia contenida}, we know that $\mathcal{P}\left(  C\right)  \cap
M_{i}\subseteq M_{j},$ which entails that $C_{0},C_{1}\in M_{j},$ hence
$C_{0},C_{1}\in R_{j}.$ This finishes the proof of the claim.

\qquad\qquad\qquad

Fix $\left\langle R_{i}\mid i\leq n\right\rangle $ as above. \ Define
$R_{n+1}=(M_{n}\cap%
{\textstyle\bigcup\limits_{i\leq n}}
R_{i})\cup\{A_{n+1}\setminus%
{\textstyle\bigcup\limits_{i\leq n}}
R_{i}\}.$ It is easy to see that $\left\langle R_{i}\mid i\leq
n+1\right\rangle $ is as desired.
\end{proof}

\qquad\ \ \qquad\ \ 

We now prove the final result of this section, which will enable us to
transfer certain names across elementary submodels.

\begin{proposition}
[\textsf{CH}]Let $\left\langle M_{0},...,M_{n}\right\rangle $ be a $\delta
$-increasing sequence of models from \textsf{Sub}$\left(  \kappa\right)  $,
$l\in\omega$ and $\mathcal{P}=\left\langle P_{i}\mid i\leq n\right\rangle $ a
coherent sequence of partitions for $\left\langle M_{0},...,M_{n}\right\rangle
$ where $P_{i}\subseteq\left[  \omega_{l}\right]  ^{\leq\omega}.$ There is a
bijection $\triangle:\omega_{l}\longrightarrow\omega_{l}$ with the following
properties: \label{prop permutacion}

\begin{enumerate}
\item If $A\in%
{\textstyle\bigcup\limits_{i\leq n}}
P_{i},$ then $\triangle\left[  A\right]  \in M_{n}$ and $\triangle
\upharpoonright A$ is order preserving.

\item If $A\in P_{n},$ then $\triangle\upharpoonright A$ is the identity mapping.
\end{enumerate}
\end{proposition}

\begin{proof}
Take an enumeration $\left\{  A_{0},...,A_{m}\right\}  $ of all elements of $%
{\textstyle\bigcup\limits_{i<n}}
P_{i}$ that are not in $P_{n}$ and denote $Y=%
{\textstyle\bigcup}
{\textstyle\bigcup\limits_{i\leq n}}
P_{i}.$ Choose $\beta<\delta_{n}$ such that \textsf{OT}$\left(  Y\right)
<\beta$ and $Y\cap\omega_{1}\subseteq\beta.$ For each $k\leq m,$ denote
$\gamma_{k}=$ \textsf{OT}$\left(  A_{k}\right)  $ and $e_{k}:A_{k}%
\longrightarrow\gamma_{k}$ the (unique) order isomorphism. Define
$\triangle_{k}:A_{k}\longrightarrow\omega_{1}$ where $\triangle_{k}\left(
\alpha\right)  =\beta\left(  k+1\right)  +e_{k}\left(  \alpha\right)  $ and
note that \textsf{im}$\left(  \triangle_{k}\right)  =[\beta\left(  k+1\right)
,\beta\left(  k+1\right)  +\gamma_{k}),$ which belongs to $M_{n}$ since
$\beta,\gamma_{k}\in M_{n}.$ Moreover, note that if $k\neq r,$ then
\textsf{im}$\left(  \triangle_{k}\right)  \cap$ \textsf{im}$\left(
\triangle_{r}\right)  =\emptyset$ and if $A\in P_{n},$ then $A\cap\omega
_{1}\subseteq\beta$, while \textsf{im}$\left(  \triangle_{k}\right)  \cap
\beta=\emptyset$ for every $k\leq m,$ so \textsf{im}$\left(  \triangle
_{k}\right)  $ and $A$ are disjoint. In this way, we can extend $%
{\textstyle\bigcup\limits_{k\leq m}}
\triangle_{k}$ to a permutation of $\omega_{l}$ that fixes every element of
$P_{n}.$
\end{proof}

\section{Forcing multiple $\mathfrak{d}$-pathways \label{forzando pathways}}

We now apply the results from the previous section to establish the existence
of multiple $\mathfrak{d}$-pathways in certain random and Cohen models. Having
that goal in mind, we need to introduce a couple of definitions.

\begin{definition}
Let $\mathbb{P}$ be a partial order and $\dot{a},\dot{b}$ two $\mathbb{P}%
$-names. We say that $\dot{a}$ and $\dot{b}$ \emph{are equivalent }if
$\mathbb{P}\Vdash$\textquotedblleft$\dot{a}=\dot{b}$\textquotedblright.
\end{definition}

\qquad\qquad\qquad\ \ \ 

We now introduce the main definition of this Section.

\begin{definition}
Let $\mathbb{P}$ be a partial order. We say that $\mathbb{P}$ \emph{has the
transformation property }if for every large enough regular $\kappa$,
$\left\langle M_{0},...,M_{n}\right\rangle $ a $\delta$-increasing sequence of
models in \textsf{Sub}$\left(  \kappa\right)  $ where $\mathbb{P}\in M_{i}$
for every $i\leq n$, $\dot{a}_{0}\in M_{0},...,\dot{a}_{n}\in M_{n}$ that are
$\mathbb{P}$-names for subsets of $\omega,$ there is an automorphism
$H:\mathbb{P\longrightarrow P}$ such that:

\begin{enumerate}
\item $H\left(  \dot{a}_{n}\right)  =\dot{a}_{n}.$

\item For every $i<n$, we have that $H(\dot{a}_{i})$ is equivalent to a
$\mathbb{P}$-name that is in $M_{n}.$
\end{enumerate}
\end{definition}

It is worth pointing out that the automorphism $H$ is not required to be in
$M_{n}.$ Note that we obtain an equivalent definition if we take each $\dot
{a}_{i}$ to be a $\mathbb{P}$-name for an element of $\omega^{\omega}$, or
even a $\mathbb{P}$-name for a pair of functions from $\omega^{\omega}$
(instead of a $\mathbb{P}$-name for a subset of $\omega$).

\qquad\ \qquad\ \ \ \ \ \ \ \ \qquad\ \ \ 

We recommend the reader to consult Section \ref{Preliminares Forcing} as we
will be using the notation and results from there.

\begin{theorem}
[\textsf{CH}]Let $\mathbb{P}$ be a ccc forcing that does not add dominating
reals and has the transformation property. $\mathbb{P}$ forces that there is a
multiple $\mathfrak{d}$-pathway.
\end{theorem}

\begin{proof}
Choose $\mathcal{B}=\left\{  f_{\alpha}\mid\alpha\in\omega_{1}\right\}
\subseteq\omega^{\omega}$ a scale of increasing functions, $\kappa$ a large
enough regular cardinal such that $\mathbb{P\in}$ \textsf{H}$\left(
\kappa\right)  .$ Define $\mathcal{S}_{0}=\{M\mid M\in$ \textsf{Sub}$\left(
\kappa\right)  \wedge\mathcal{B},\mathbb{P}\in M\},$ which is stationary. We
claim that if $G\subseteq\mathbb{P}$ is a generic filter, then in $V\left[
G\right]  $ we will have that $\left(  \mathcal{B},\mathcal{S}\right)  $ is a
multiple $\mathfrak{d}$-pathway, where $\mathcal{S}=\{M\left[  G\right]  \mid
M\in$ $\mathcal{S}\}.$ Since $\mathbb{P}$ is ccc, $\mathcal{S}$ is forced to
be stationary, $M\left[  G\right]  \cap V=M$ and $\delta_{M\left[  G\right]
}=\delta_{M}$ for every $M\in\mathcal{S}_{0}$ (see \cite{ProperandImproper}
and \cite{NotesonForcingAxioms}).

\qquad\qquad\qquad\ \ \ 

Let $p\in\mathbb{P},$ $\left\langle M_{0},...,M_{n}\right\rangle $ a $\delta
$-increasing sequence of models in $\mathcal{S}_{0}$, $\dot{x}_{0}\in
M_{0},...,\dot{x}_{n}\in M_{n}$ that are $\mathbb{P}$-names for elements of
$\omega^{\omega}$ and $\dot{R}\in M_{n}$ a name for a projective and
$\leq^{\ast}$-adequate relation. We will now apply the transformation property
of $\mathbb{P}.$ For $i<n,$ let $\dot{a}_{i}=\dot{x}_{i}$, let $\dot{u}$ be a
$\mathbb{P}$-name for an element of $\omega^{\omega}$ that is forced to code
$\dot{R}$ (this is possible since $\dot{R}$ is forced to be a projective
relation) and let $\dot{a}_{n}$ be a $\mathbb{P}$-name for the pair $(\dot
{u},\dot{x}_{n}).$ By the transformation property, we can find an automorphism
$H:\mathbb{P\longrightarrow P}$ such that $H\left(  \dot{x}_{i}\right)  $ is
equivalent to a name in $M_{n}$ for every $i\leq n,$ $H\left(  \dot{x}%
_{n}\right)  =\dot{x}_{n}$ and $H(\dot{R})=\dot{R}.$ Since $\dot{R}$ is forced
to be $\leq^{\ast}$-adequate, there is $\dot{g}\in M_{n}$ a $\mathbb{P}$-name
that is forced to be an $\dot{R}$-control for $\left(  H\left(  \dot{x}%
_{0}\right)  ,...,H\left(  \dot{x}_{n}\right)  \right)  .$ Since $\mathbb{P}$
does not add dominating reals and is ccc, we know that $H\left(  p\right)
\Vdash$\textquotedblleft$f_{\delta_{n}}\nleq^{\ast}\dot{g}$\textquotedblright.
Therefore, we know that $H\left(  p\right)  \Vdash$\textquotedblleft$H(\dot
{R})\left(  H\left(  \dot{x}_{0}\right)  ,...,H\left(  \dot{x}_{n}\right)
,f_{\delta_{n}}\right)  $\textquotedblright\ and since $H$ is an isomorphism,
with the aid of Proposition \ref{Prop formula isomorfismo}, we conclude that
$p\Vdash$\textquotedblleft$\dot{R}\left(  \dot{x}_{0},...,\dot{x}%
_{n},f_{\delta_{n}}\right)  $\textquotedblright.
\end{proof}

\qquad\ \ \qquad\ \ 

Both Cohen and random forcings are ccc and do not add dominating reals (random
forcing does not even add unbounded reals, see \cite{Barty}). Our next goal is
to show that they have the transformation property. For the remainder of the
section, $l$ will denote a natural number, $\mathbb{P}\left(  I\right)  $ will
be either Cohen or random forcing (where $I\subseteq\omega_{l}$) and $\kappa$
a large enough regular cardinal.

\begin{proposition}
[\textsf{CH}]Let $M,N\in$ \textsf{Sub}$\left(  \kappa\right)  $ with
$\delta_{M}\leq\delta_{N},$ $I\in M\cap\left[  \omega_{l}\right]  ^{\omega}$
and $\dot{a}\in M$ a $\mathbb{P}\left(  I\right)  $-name for a subset of
$\omega.$ If $\triangle:\omega_{l}\longrightarrow\omega_{l}$ a permutation for
which there is $P\in M$ a finite partition of $I$ such that for every $A\in P$
we have that $\triangle\upharpoonright A$ is order preserving and
$\triangle\left[  A\right]  \in N,$ then $\triangle_{\ast}\left(  \dot
{a}\right)  $ is equivalent to a name in $N.$ \label{prop mandar nombres}
\end{proposition}

\begin{proof}
Take an enumeration $P=\left\{  P_{i}\mid i\leq n\right\}  $ and choose
$\beta<\delta_{M}$ a limit ordinal larger than \textsf{OT}$\left(  I\right)
.$ For each $i\leq n,$ \ denote $\gamma_{i}=$ \textsf{OT}$\left(
P_{i}\right)  $ and $e_{i}:P_{i}\longrightarrow\gamma_{i}$ the only
isomorphism. Since $\triangle$ is order preserving in each $P_{i},$ we know
that $\gamma_{i}$ is isomorphic to $\triangle\left[  P_{i}\right]  $ as well.
Let $\widehat{e}_{i}:\triangle\left[  P_{i}\right]  \longrightarrow\gamma_{i}$
be the only isomorphism. Note that $e_{i}\in M$ and $\widehat{e}_{i}\in N.$ We
now define the function:

\qquad\ \ \ 

\hfill%
\begin{tabular}
[c]{lll}%
$h:I\longrightarrow\omega_{1}$ &  & $g:\triangle\left[  I\right]
\longrightarrow\omega_{1}$\\
\multicolumn{1}{c}{} & \multicolumn{1}{c}{} & \multicolumn{1}{c}{}\\
\multicolumn{3}{c}{Such that for every $\alpha\in P_{i}:$}\\
\multicolumn{1}{c}{} & \multicolumn{1}{c}{} & \multicolumn{1}{c}{}\\
$h\left(  \alpha\right)  =\beta i+e_{i}\left(  \alpha\right)  $ &  & $g\left(
\triangle\left(  \alpha\right)  \right)  =\beta i+\widehat{e}_{i}\left(
\triangle\left(  \alpha\right)  \right)  $\\
&  & \ \ \ \ \ \ \ \ \ \ \ \ \ $=\beta i+e_{i}\left(  \alpha\right)  $\\
&  & \ \ \ \ \ \ \ \ \ \ \ \ \ $=h\left(  \alpha\right)  $%
\end{tabular}
\qquad\ \hfill

\qquad\ \qquad\ \ \ \ \ \ \qquad\ \ \ \ \ 

Clearly $h=g\circ\triangle$ and \textsf{im}$\left(  h\right)  =$
\textsf{im}$\left(  g\right)  .$ We have the isomorphisms $h_{\ast}:$
$\mathbb{P}\left(  I\right)  \longrightarrow\mathbb{P}\left(  h\left[
I\right]  \right)  $ and $g_{\ast}:\mathbb{P}\left(  \triangle\left[
I\right]  \right)  \longrightarrow\mathbb{P}\left(  h\left[  I\right]
\right)  $. Note that $h,h_{\ast}\in M$ and $g,g_{\ast}\in N.$ Denote $\dot
{b}=h_{\ast}\left(  \dot{a}\right)  ,$ which is a $\mathbb{P}\left(  h\left[
I\right]  \right)  $-name. Since $h\left[  I\right]  $ is a countable subset
of $\omega_{1},$ since $\mathbb{P}\left(  h\left[  I\right]  \right)  $ is
ccc, we can code $\dot{b}$ as an element of $M\cap$\textsf{H}$\left(
\omega_{1}\right)  $ and by Lemma \ref{Lema contencion Hw1}, we conclude that
$\dot{b}\in N.$ In this way, we know that $\dot{c}=g_{\ast}^{-1}(\dot{b})$ is
a $\mathbb{P}\left(  \triangle\left[  I\right]  \right)  $-name that is in
$N.$ In this way, in order to prove that $\triangle_{\ast}\left(  \dot
{a}\right)  $ is in $N,$ it is enough to show that it is equal to $\dot{c}.$
Let $p\in\mathbb{P}\left(  \triangle\left[  I\right]  \right)  $ and
$n\in\omega.$ Using Propositions \ref{Prop formula isomorfismo} and
\ref{prop estrellita}, we obtain the following:

\qquad\qquad\ \ \ \ \ 

\hfill%
\begin{tabular}
[c]{lll}%
$p\Vdash$\textquotedblleft$n\in\dot{c}$\textquotedblright &
$\longleftrightarrow$ & $p\Vdash$\textquotedblleft$n\in g_{\ast}^{-1}(\dot
{b})$\textquotedblright\\
& $\longleftrightarrow$ & $g_{\ast}\left(  p\right)  \Vdash$\textquotedblleft%
$n\in\dot{b}$\textquotedblright\\
& $\longleftrightarrow$ & $g_{\ast}\left(  p\right)  \Vdash$\textquotedblleft%
$n\in h_{\ast}\left(  \dot{a}\right)  $\textquotedblright\\
& $\longleftrightarrow$ & $h_{\ast}^{-1}\circ g_{\ast}\left(  p\right)
\Vdash$\textquotedblleft$n\in\dot{a}$\textquotedblright\\
& $\longleftrightarrow$ & $\left(  h^{-1}\circ g\right)  _{\ast}\left(
p\right)  \Vdash$\textquotedblleft$n\in\dot{a}$\textquotedblright\\
& $\longleftrightarrow$ & $\triangle_{\ast}^{-1}\left(  p\right)  \Vdash
$\textquotedblleft$n\in\dot{a}$\textquotedblright\\
& $\longleftrightarrow$ & $p\Vdash$\textquotedblleft$n\in\triangle_{\ast
}\left(  \dot{a}\right)  $\textquotedblright%
\end{tabular}
\qquad\ \qquad\hfill

\qquad\qquad\qquad\qquad\ \ \ 

We conclude that $\triangle_{\ast}\left(  \dot{a}\right)  $ and $\dot{c}$ are
equivalent. Since $c\in N$, the proof is finished.
\end{proof}

\qquad\ \ \ \qquad\ \ \ \ 

We can finally prove:

\begin{proposition}
[\textsf{CH}]$\mathbb{P}\left(  \omega_{l}\right)  $ has the transformation
property.\footnote{Stepr\={a}ns recently found a more direct proof of this
result. It will appear in a forthcoming book he is writing with Raghavan.}
\end{proposition}

\begin{proof}
Let $\kappa$ be a regular large enough cardinal, $\left\langle M_{0}%
,...,M_{n}\right\rangle $ a $\delta$-increasing sequence of models from
\textsf{Sub}$\left(  \kappa\right)  $ and $\dot{a}_{0}\in M_{0},...,\dot
{a}_{n}\in M_{n}$ be $\mathbb{P}\left(  \omega_{l}\right)  $ names for subsets
of $\omega.$ For every $i\leq n,$ find $A_{i}\in M_{i}\cap\left[  \omega
_{l}\right]  ^{\omega}$ such that $\dot{a}_{i}$ is a $\mathbb{P}\left(
A_{i}\right)  $-name. We now use Lemma \ref{Lema particion que cubre} to
summon $\mathcal{P}=\left\langle P_{i}\mid i\leq n\right\rangle $ a coherent
sequence of partitions for $\left\langle M_{0},...,M_{n}\right\rangle $ such
that $A_{i}\subseteq\cup P_{i}$ for every $i\leq n.$ Denote $I_{i}=\cup P_{i}$
for every $i\leq n.$ Clearly $I_{i}\in M_{i}$ and $\dot{a}_{i}$ is a
$\mathbb{P}\left(  I_{i}\right)  \,$-name. We now invoke Proposition
\ref{prop permutacion} to find a permutation $\triangle:\omega_{l}%
\longrightarrow\omega_{1}$ such that is the identity in every element of
$P_{n}$ and for every $B\in%
{\textstyle\bigcup}
P_{i}$ it is the case that $\triangle\left[  B\right]  \in M_{n}$ and
$\triangle\upharpoonright B$ is order preserving. Finally, by Proposition
\ref{prop mandar nombres}, we know that $\triangle_{\ast}\left(  \dot{a}%
_{i}\right)  $ is equivalent to a name in $M_{n}$ for every $i\leq n.$
Moreover, since $\triangle\upharpoonright I_{n}$ is the identity, we get that
$\triangle_{\ast}\left(  \dot{a}_{n}\right)  =\dot{a}_{n}.$
\end{proof}

\qquad\ \ \qquad\ \ 

We can now conclude:

\begin{corollary}
[\textsf{CH}]Let $l<\omega.$ Both $\mathbb{C}\left(  \omega_{l}\right)  $ and
$\mathbb{B}\left(  \omega_{l}\right)  $ force that there is a multiple
$\mathfrak{d}$-pathway.
\end{corollary}

\qquad\ \ \qquad

In particular:

\begin{theorem}
[\textsf{CH}]Let $l<\omega.$ $\mathbb{B}\left(  \omega_{l}\right)  $ force
that there is a strong \textsf{P}-point and a Gruff ultrafilter.
\end{theorem}

\qquad\qquad\ \ 

Of course this is also true for Cohen forcing, but it is not new since the
existence of a Gruff ultrafilter and a strong \textsf{P}-point follow from
$\mathfrak{d=c}$ (see \cite{GruffUltrafilters} and \cite{CanjarFilters}).

\section{Open Questions}

We now list some questions we do not know how to solve. The most important one
is the following:

\begin{problem}
Are there \textsf{P}-points (Gruff ultrafilters, strong \textsf{P}-points) in
every model obtained by adding any number of random reals to a model of
\textsf{CH}?
\end{problem}

\qquad\ \ \qquad

It would be enough to provide a positive answer to the following:

\begin{problem}
Does \textsf{CH }imply that $\mathbb{C}\left(  \kappa\right)  $ and
$\mathbb{B}\left(  \kappa\right)  $ have the transformation property for any
cardinal $\kappa$?
\end{problem}

\qquad\ \ \qquad\ \ \ 

In \cite{PFiltersCohenRandomLaver} the first author proved that there will be
an ultrafilter that does not contain a nowhere dense \textsf{P}-subfilter
(equivalently, $\omega^{\ast}$ can not be covered by nowhere dense
\textsf{P}-sets) after adding $\omega_{2}$ Cohen reals to a model of
\textsf{CH}$+$ $\square_{\omega_{1}}$. Building on the ideas developed in this
paper and in \cite{PFiltersCohenRandomLaver}, we were able to construct such
an ultrafilter after adding fewer than $\aleph_{\omega}$ Cohen reals over a
model of \textsf{CH, }this result will appear in another paper.

\begin{acknowledgement}
We would like to thank Michael Hru\v{s}\'{a}k and Juris Stepr\={a}ns for
several helpful discussions related to the topic of this paper. We would also
like to thank the referee for their comments, which greatly improved the paper.
\end{acknowledgement}

\bibliographystyle{plain}
\def\cprime{$'$}

\qquad\qquad\qquad\ \ \ \ \ \ \ \ \ \ \ \ \ \ \ \ \ \ \ \qquad\qquad\qquad\ \ \ 

Alan Dow

Department of Mathematics and Statistics, UNC Charlotte.

adow@charlotte.edu

\qquad\qquad\qquad\ \ \ \ \ \ \ \ \ \ \ \ \qquad\ \ \ \ \ \qquad
\qquad\ \ \ \qquad\qquad\qquad\qquad\ \ \ \ \ \ 

Osvaldo Guzm\'{a}n

Centro de Ciencias Matem\'{a}ticas, UNAM.

oguzman@matmor.unam.mx

\end{document}